\documentclass[a4paper,11pt]{article}

\usepackage[utf8]{inputenc}

\usepackage[T1]{fontenc}
\usepackage{fouriernc}

\usepackage{amsmath}
\usepackage{amssymb}
\usepackage{amsthm}
\usepackage{enumitem}
\usepackage{graphicx}
\usepackage{xcolor}

\setlist{nolistsep}

\usepackage[margin=20mm,includehead,includefoot]{geometry}

\usepackage{fancyhdr}
\fancypagestyle{plain}{
\fancyhf{}
\fancyfoot[C]{\textbf{\thepage}}

}

\definecolor{cornellrot}{RGB}{179,27,27}
\usepackage[colorlinks=true,linkcolor=cornellrot,citecolor=cornellrot,urlcolor=cornellrot]{hyperref}

\usepackage{tikz}
\newcommand{\circled}[1]{\tikz[baseline=-3.8pt]{\node[shape=circle,draw,inner sep=0.5pt] (char) {\scriptsize #1};}}

\makeatletter
\newcommand*\rel@kern[1]{\kern#1\dimexpr\macc@kerna}
\newcommand*\widebar[1]{\begingroup\def\mathaccent##1##2{\rel@kern{0.8}\overline{\rel@kern{-0.8}\macc@nucleus\rel@kern{0.2}}\rel@kern{-0.2}}
\macc@depth\@ne\let\math@bgroup\@empty \let\math@egroup\macc@set@skewchar\mathsurround\z@ \frozen@everymath{\mathgroup\macc@group\relax}\macc@set@skewchar\relax\let\mathaccentV\macc@nested@a\macc@nested@a\relax111{#1}\endgroup}
\makeatother

\renewcommand{\Re}{\operatorname{Re}}
\renewcommand{\Im}{\operatorname{Im}}

\newcommand {\divides}               {\mathbin{|}}
\newcommand {\hooktwoheadrightarrow} {\hookrightarrow\hspace{-1.8ex} \rightarrow}

\renewcommand {\d} [1] {\ensuremath{\operatorname{d}\!{#1}}}

\DeclareMathOperator {\arcosh} {arcosh}
\DeclareMathOperator {\Aut}    {Aut}
\DeclareMathOperator {\bnd}    {bnd}
\DeclareMathOperator {\BS}     {BS}
\DeclareMathOperator {\card}   {card}

\DeclareMathOperator {\isom}   {Isom}

\DeclareMathOperator {\supp}   {supp}

\newcommand {\dhatt}  {d_{\widehat{\textit{T}}}\,}
\newcommand {\dhyp}   {d_{\mathbb{H}}}
\newcommand {\ddhyp}  {d_{\Gamma_{v}}}
\newcommand {\dtree}  {d_{T}}
\newcommand {\dwordS} {d_{S}}

\newcommand {\deucl}  {d_{\mathbb{R}^{2}}}

\newcommand {\set}  [2] {\ensuremath{\{\, #1 \,:\, #2 \,\}}}
\newcommand {\sset} [1] {\ensuremath{\{\, #1 \,\}}}
\newcommand {\pres} [2] {\ensuremath{\langle\, #1 \,:\, #2 \,\rangle}}
\newcommand {\sgp}  [1] {\ensuremath{\langle\, #1 \,\rangle}}
\newcommand {\nsgp} [1] {\ensuremath{\langle\hspace*{-1.6pt}\langle\, #1 \,\rangle\hspace*{-1.6pt}\rangle}}

\newtheoremstyle{graz}{10pt}{5pt}{\normalfont\it}{}{\normalfont\bfseries}{}{6pt}{\textbf{\thmname{#1}\thmnumber{ #2}\thmnote{ (#3)}}}
\theoremstyle{graz}

\newtheorem {rem} {Remark} [section]

\newtheorem {lem}  [rem] {Lemma}

\newtheorem {defn} [rem] {Definition}
\newtheorem {thm}  [rem] {Theorem}

\newtheorem {exa}  [rem] {Example}

\AtBeginDocument{
  \DeclareSymbolFont{AMSb}{U}{msb}{m}{n}
  \DeclareSymbolFontAlphabet{\mathbb}{AMSb}}
  
\begin{document}

\allowdisplaybreaks


\title{Random walks on Baumslag--Solitar groups}

\author{
Johannes Cuno\thanks{Research supported by the Austrian Science Fund (FWF): W1230-N13 and P24028-N18, the Canada Research Chairs Program, and the European Research Council (ERC): No 725773 ``GroIsRan''.}~~and
Ecaterina Sava-Huss\thanks{Research supported by the Austrian Science Fund (FWF): J3575-N26.}
}

\date{Revised version --- November 15, 2017}


\maketitle

\begin{abstract}
\noindent 
We consider random walks on non-amenable Baumslag--Solitar groups $\BS(p,q)$ and describe their Poisson--Furstenberg boundary. The latter is a probabilistic model for the long-time behaviour of the random walk. In our situation, we identify it in terms of the space of ends of the Bass--Serre tree and the real line using Kaimanovich's strip criterion.
\end{abstract}


\section{Introduction}
\hypertarget{sec:introduction}{}

For any two non-zero integers $p$ and $q$ the Baumslag--Solitar group $\BS(p,q)$ is given by the presentation $\BS(p,q)=\pres{a,b}{ab^{p}=b^{q}a}$. These groups were introduced by Baumslag and Solitar in \cite{bs-1962}, who identified $\BS(2,3)$ as the first example of a two-generator one-relator non-Hopfian group. We consider random walks on $\BS(p,q)$. Each of them is driven by a probability measure $\mu$ whose support generates $\BS(p,q)$ as a semigroup. The random walk starts at the identity element and proceeds with independent $\mu$-distributed increments $X_{1},X_{2},\dotsc$ being multiplied from the right to the current state.

The Poisson--Furstenberg boundary was introduced by Furstenberg in \cite{furstenberg-1963} and \cite{furstenberg-1971}. It is a probabilistic model for the long-time behaviour of the random walk and simultaneously provides a way to represent all bounded harmonic functions on the state space. In \cite[Theorem~5.1]{kaimanovich-1991}, Kaimanovich considered random walks on $\BS(1,2)$. Under the assumption of finite first moment, he identified their Poisson--Furstenberg boundary geometrically. In particular, he showed that the latter is trivial if the random walk has no vertical drift, i.\,e.~the expected exponent sum of the increments with respect to the generator $a$ is equal to zero.

For random walks on non-amenable groups the situation is different. As long as they are irreducible, their Poisson--Furstenberg boundary can never be trivial. This motivates the present paper, in which we study random walks on non-amenable Baumslag--Solitar groups. It is organised as follows. In Section \ref{sec:bsgroups}, we discuss some algebraic and geometric properties of Baumslag--Solitar groups $\BS(p,q)$ with $1\leq p<q$. We explain how these groups can be understood through their projections to the Bass--Serre tree~$T$ and the hyperbolic plane $\mathbb{H}$. Afterwards, we recall the construction of the space of ends $\partial T$ and the hyperbolic boundary $\partial\mathbb{H}$, which contains the real line~$\mathbb{R}$ as a subset. These spaces shall later be used to associate a geometric boundary to $\BS(p,q)$. In Section \ref{sec:randomwalks}, we turn to random walks on groups. We outline some results about the Poisson--Furstenberg boundary and then state Kaimanovich's strip criterion, which is an important tool to identify this boundary geometrically.

In Section \ref{sec:identification}, we study random walks on $\BS(p,q)$ with finite first moment. We consider the pointwise projections of the random walk to $T$ and $\mathbb{H}$. If the random walk has negative vertical drift, then the projection to $\mathbb{H}$ converges almost surely to a random element in $\mathbb{R}$. For the projection to $T$, we do not need to distinguish between different vertical drifts; as soon as $1<p<q$, it converges almost surely to a random element in $\partial T$. We thus endow $\partial T$ (or even $\partial T\times\mathbb{R}$) with the Borel $\sigma$-algebra $\mathcal{B}_{\partial T}$ (or $\mathcal{B}_{\partial T\times\mathbb{R}}$) and the hitting measure $\nu_{\partial T}$ (or~$\nu_{\partial T\times\mathbb{R}}$). Finally, Kaimanovich's strip criterion shows that the resulting probability space is isomorphic to the Poisson--Furstenberg boundary.

Up to and including Section~\ref{sub:cttbothp}, we assume that the two non-zero integers $p$ and $q$ satisfy $1\leq p<q$. Then, we restrict ourselves to the non-amenable subcase $1<p<q$. In the \hyperlink{sec:appendix}{appendix}, we explain how to obtain similar results for the remaining non-amenable cases $1<p<-q$ and $1<p=|q|$. Our main result is the following.

\begin{thm}[``identification theorem'']\label{thm:identification}
Let $Z=(Z_{0},Z_{1},\dotsc)$ be a random walk on a non-amenable Baumslag--Solitar group $G=\BS(p,q)$ with $1<p<q$ and increments $X_{1},X_{2},\dotsc$ of finite first moment. Depending on the vertical drift $\delta$, we distinguish three cases: 
\begin{enumerate}
\item If $\delta>0$, then the Poisson--Furstenberg boundary is isomorphic to $(\partial T ,\mathcal{B}_{\partial T},\nu_{\partial T})$ endowed with the boundary map $\bnd_{\partial T}:\Omega\to\partial T$.
\item If $\delta<0$, then it is isomorphic to $(\partial T\times\mathbb{R},\mathcal{B}_{\partial T\times\mathbb{R}},\nu_{\partial T\times\mathbb{R}})$ endowed with $\bnd_{\partial T\times\mathbb{R}}:\Omega\to\partial T\times\mathbb{R}$.
\item If $\delta=0$ and $\ln(A_{X_{1}})$ has finite second moment and there is an $\varepsilon>0$ such that $\ln(1+|B_{X_{1}}|)$ has finite $(2+\varepsilon)$-th moment, then it is isomorphic to~$(\partial T ,\mathcal{B}_{\partial T},\nu_{\partial T})$ endowed with $\bnd_{\partial T}:\Omega\to\partial T$.
\end{enumerate}
\end{thm}

Note that the driftless case is a little more subtle and requires additional assumptions on the moments. Here, the terms $A_{g}$ and $B_{g}$ denote the imaginary and real part of the projection of an element $g\in G$ to the hyperbolic plane $\mathbb{H}$. The two assumptions are certainly satisfied if $X_{1}$ has finite $(2+\varepsilon)$-th moment. Further details can be found at the beginning of Section~\ref{sub:cttbothp}.


\paragraph*{Acknowledgements}

We would like to thank Wolfgang Woess for suggesting this problem to us and supporting us with references and ideas while the research was carried out. Moreover, we are grateful to Vadim Kaimanovich and the anonymous reviewer, both of whom made valuable suggestions that led to a substantial improvement of the present paper.


\section{Baumslag--Solitar groups}
\label{sec:bsgroups}


\subsection{Amenability of Baumslag--Solitar groups}
\label{sub:aobsg}

The structure of Baumslag--Solitar groups can be studied by means of HNN extensions. Indeed, $\BS(p,q)$ is precisely the HNN extension $\mathbb{Z}\ast_{\varphi}$ with isomorphism $\varphi:p\mathbb{Z}\to q\mathbb{Z}$ given by $\varphi(p):=q$. This allows us to use the respective machinery, such as Britton's lemma, see \cite{britton-1963}, which implies that a freely reduced non-empty word $w$ over the letters $a$ and $b$ and their formal inverses can only represent the identity element $1\in\BS(p,q)$ if it contains $ab^{r}a^{-1}$ with $p\divides r$ or $a^{-1}b^{r}a$ with $q\divides r$ as a subword. In particular, if neither $|p|=1$ nor $|q|=1$, then the elements $x:=a$ and $y:=bab^{-1}$ generate a non-abelian free subgroup and $\BS(p,q)$ is non-amenable. On the other hand, if $|p|=1$ or $|q|=1$, a simple calculation shows that the normal subgroup $\nsgp{b}\trianglelefteq\BS(p,q)$ is abelian with quotient isomorphic to $\mathbb{Z}$. In this case, $\BS(p,q)$ is solvable and therefore amenable. As we will address briefly in Section~\ref{sub:sratpfb}, the distinction between these two cases is of importance when working with random walks.


\subsection{Projection to the Bass--Serre tree}

Assume first that $1\leq p<q$. The Cayley graph $\Gamma$ of the group $G:=\BS(p,q)$ with respect to the standard generators $a$ and $b$ is the directed multigraph with vertex set $G$, edge set $G\times\sset{a,b}$, source function $s:G\times\sset{a,b}\to G$ given by $s(g,x):=g$, and target function $t:G\times\sset{a,b}\to G$ given by $t(g,x):=gx$. Every directed multigraph can be converted into a simple graph by ignoring the direction and the multiplicity of the edges and deleting the loops. For the purpose of this paper, it is sufficient to think of $\Gamma$ as a simple graph, and we shall tacitly do so.

\begin{defn}[``levels'']
\label{defn:levelmaps}
Let $\lambda:\sset{a,b}\to\mathbb{Z}$ be the map given by $\lambda(a):=1$ and $\lambda(b):=0$. It follows from von Dyck's theorem, see e.\,g.~\cite[\S 1.1.3]{cz-1993}, that this map can be uniquely extended to a group homomorphism $\lambda:G\to\mathbb{Z}$. We think of it as a level function.
\end{defn}

Consider the illustration of $\Gamma$ in Figure~\ref{fig:cayley}. Intuitively speaking, we may look at it from the side to see the associated Bass--Serre tree. Formally, let $B:=\sgp{b}\leq G$ and let $T$ be the graph with vertex set $G/B=\set{gB}{g\in G}$ and edge set $\set{\{\,gB,gaB\,\}}{g\in G}$. This graph is actually a tree; it is connected and, by Britton's lemma, it does not contain any cycle. We use the symbol $\pi_{T}$ to denote the canonical projection to the cosets, i.\,e.~the map $\pi_{T}:G\to G/B$ given by $\pi_{T}(g):=gB$.

\begin{rem}
\label{rem:neighbours}
Since $\lambda(b)=0$, the level function is well-defined on the vertices of $T$. It is not hard to see that every vertex of $T$ has exactly $p+q$ neighbours; $p$ of them are one level below and $q$ of them are one level above the vertex.
\end{rem}

\begin{figure}
\begin{center}
  \includegraphics{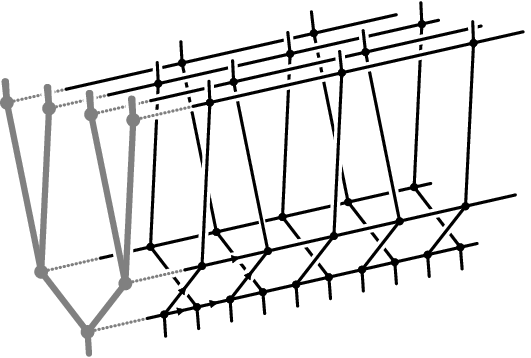}
  \begin{picture}(0,0)
    \put(-180.6,30){$a$}
    \put(-149,36.5){$a$}
    \put(-174,12){$b$}
    \put(-158,15.5){$b$}
    \put(-157,48.5){$b$}
  \end{picture}
\end{center}
\caption{The Cayley graph $\Gamma$ of $\BS(1,2)$ with respect to the standard generators $a$ and $b$.}
\label{fig:cayley}
\end{figure}
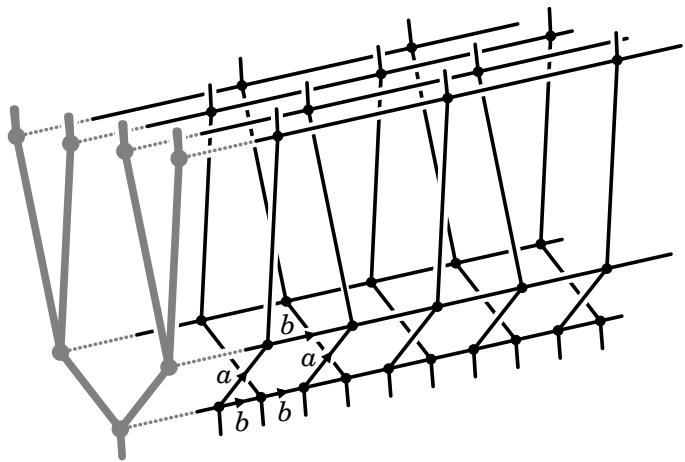


\subsection{Projection to the hyperbolic plane}
\label{sub:ptthp}

The second projection captures the information that is obtained by looking at $\Gamma$ from the front. It is convenient to describe it in terms of the hyperbolic plane. So, let $\mathbb{H}$ be the hyperbolic plane as per the Poincar\'e half-plane model, i.\,e.~$\mathbb{H}:=\set{z\in\mathbb{C}}{\Im(z)>0}$, endowed with the standard metric
\begin{equation*}
  \dhyp(z_{1},z_{2}):=\ln\left(\frac{|z_{1}-\widebar{z_{2}}|+|z_{1}-z_{2}|}{|z_{1}-\widebar{z_{2}}|-|z_{1}-z_{2}|}\right)=\arcosh\left(1+\frac{|z_{1}-z_{2}|^{2}}{2\cdot\Im(z_{1})\Im(z_{2})}\right)\,.
\end{equation*}
The isometry group $\isom(\mathbb{H})$ consists of all maps $\varphi:\mathbb{H}\to\mathbb{H}$ of the form 
\begin{equation*}
  \underbrace{\varphi(z)=\frac{\alpha z+\beta}{\gamma z+\delta}}_{\circled{1}}\quad\text{or}\quad\underbrace{\varphi(z)=\frac{\alpha\cdot(-\widebar{z})+\beta}{\gamma\cdot(-\widebar{z})+\delta}}_{\circled{2}}\quad\text{with}\quad\alpha,\beta,\gamma,\delta\in\mathbb{R}\,,\,\alpha\delta-\beta\gamma>0\,,
\end{equation*}
see e.\,g.~\cite[Theorem 7.4.1]{beardon-1983}. For the time being, we shall only work with the orientation-preserving isometries \circled{1}. The orientation-reversing ones \circled{2} will later be of relevance, in Section~\ref{sub:absiothp} of the \hyperlink{sec:appendix}{appendix}. Let $\pi_{\isom(\mathbb{H})}:\sset{a,b}\to\isom(\mathbb{H})$ be the map given by $\pi_{\isom(\mathbb{H})}(a):=\big(z\mapsto\frac{q}{p}\cdot z\big)$ and $\pi_{\isom(\mathbb{H})}(b):=(z\mapsto z+1)$. As in Definition~\ref{defn:levelmaps}, it follows from von Dyck's theorem that this map can be uniquely extended to a group homomorphism $\pi_{\isom(\mathbb{H})}:G\to\isom(\mathbb{H})$. Now, we define $\pi_{\mathbb{H}}:G\to\mathbb{H}$ by $\pi_{\mathbb{H}}(g):=\pi_{\isom(\mathbb{H})}(g)(i)$.

\begin{lem}
\label{lem:lookingfromfront}
For every $g\in G$ the point $\pi_{\mathbb{H}}(ga)\in\mathbb{H}$ is above the point $\pi_{\mathbb{H}}(g)\in\mathbb{H}$; the two points have the same real part and their distance in the hyperbolic plane is $\ell_{a}:=\ln\big(\frac{q}{p}\big)$. Similarly, the point $\pi_{\mathbb{H}}(gb)\in\mathbb{H}$ is to the right of the point $\pi_{\mathbb{H}}(g)\in\mathbb{H}$; the two points have the same imaginary part and their distance in the hyperbolic plane is $\ell_{b}:=\ln\big(\frac{3+\sqrt{5}}{2}\big)$.
\end{lem}

\vspace{-10pt}\begin{proof}
This is clear for $g=1$. Now, pick an arbitrary element $g\in G$. The points $\pi_{\mathbb{H}}(ga)\in\mathbb{H}$ and $\pi_{\mathbb{H}}(g)\in\mathbb{H}$ are obtained by evaluating $\pi_{\isom(\mathbb{H})}(g)$ at $\pi_{\mathbb{H}}(a)\in\mathbb{H}$ and $\pi_{\mathbb{H}}(1)\in\mathbb{H}$. Since $g$ can be written as a product over $a^{\pm1}$ and $b^{\pm1}$, its image $\pi_{\isom(\mathbb{H})}(g)$ is the respective composition of $\pi_{\isom(\mathbb{H})}(a^{\pm1})$ and $\pi_{\isom(\mathbb{H})}(b^{\pm1})$. Being dilations and translations, the latter preserve the relative position of any two points in $\mathbb{H}$, and so does $\pi_{\isom(\mathbb{H})}(g)$. The same argument works for the second assertion, which completes the proof.
\end{proof}

\emph{Here and throughout the present paper, we use the symbol $\mathbb{N}_{0}$ to denote the non-negative integers and the symbol $\mathbb{N}$ to denote the strictly positive ones.}

\begin{defn}[``path'', ``reduced path'']
Given a simple graph with vertex set $V$, we consider finite paths $v:\sset{0,1,\dotsc,n}\to V$, infinite paths $v:\mathbb{N}_{0}\to V$, and doubly infinite paths $v:\mathbb{Z}\to V$. In any case, being a path means that for every possible choice of $k$ the vertices $v(k)$ and $v(k+1)$ are adjacent. A path is reduced if for every possible choice of $k$ the vertices $v(k)$ and $v(k+2)$ are distinct.
\end{defn}

\begin{rem}[``discrete hyperbolic plane'']
One way to recover hyperbolic structures within the Cayley graph $\Gamma$ is the following. Fix an ascending doubly infinite path $v:\mathbb{Z}\to G/B$ in the tree $T$. Ascending refers to the level function constructed in Definition \ref{defn:levelmaps} and Remark~\ref{rem:neighbours}, and it means that for every $k\in\mathbb{Z}$ the vertex $v(k)$ is located below the following vertex $v(k+1)$. Let $G_{v}$ be the full $\pi_{T}$-preimage of the path, i.\,e.~the set consisting of all $g\in G$ such that $\pi_{T}(g)$ is contained in $v(\mathbb{Z})$. The subgraph $\Gamma_{v}\leq\Gamma$ spanned by $G_{v}$, see \circled{1} in Figure~\ref{fig:hyperbolic}, is connected so that the graph distance $\ddhyp$ is a metric. This subgraph is sometimes referred to as discrete hyperbolic plane or plane of bricks, which makes particular sense in light of the fact that the restriction $\pi_{\mathbb{H}}|\hspace{0.5pt}_{G_{v}}:G_{v}\to\mathbb{H}$ is a quasi-isometry, even a bi-Lipschitz map, between the graph~$\Gamma_{v}$ endowed with the graph distance $\ddhyp$ and the hyperbolic plane~$\mathbb{H}$ endowed with the standard metric $\dhyp$.
\end{rem}

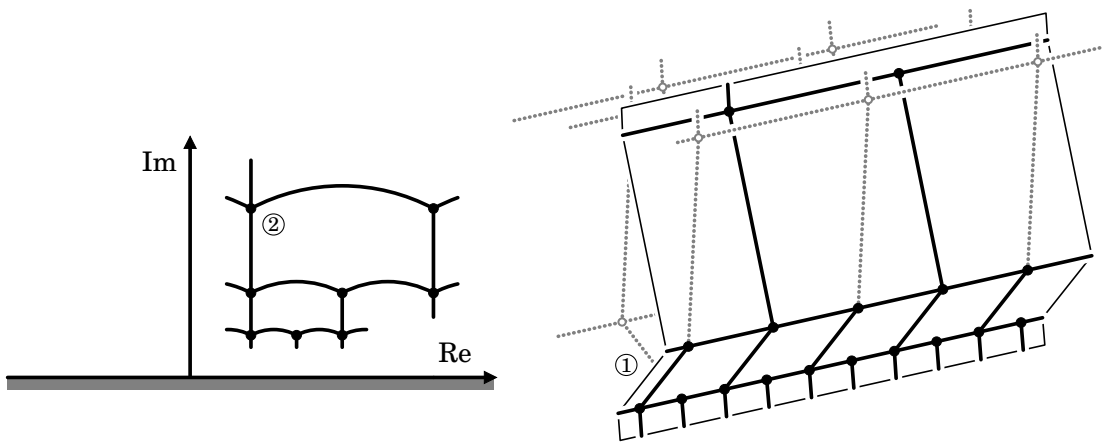
\begin{figure}
\begin{center}
  \includegraphics{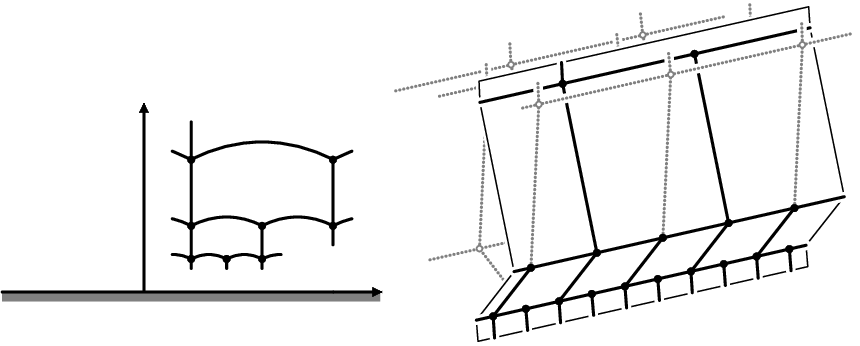}
  \begin{picture}(0,0)
    \put(-255,32){$\Re$}
    \put(-366,104){$\Im$}
    \put(-189,27){\circled{1}}
    \put(-321,80){\circled{2}}
  \end{picture}
\end{center}
\caption{A part of a discrete hyperbolic plane $\Gamma_{v}$ (right) and its projection to $\mathbb{H}$ (left).}
\label{fig:hyperbolic}
\end{figure}


\subsection{Compactifications}

Both the tree $T$ and the hyperbolic plane $\mathbb{H}$ have a natural compactification. In case of $T$, it is the \emph{end compactification}, which can be constructed as follows. Fix a base point, say $B\in G/B$, and consider the set $\widehat{T}$ of all reduced paths that start in $B$, be they finite or infinite. The endpoint map yields a one-to-one correspondence between the finite paths and the vertices $G/B$. We may therefore think of $G/B$ as a subset of $\widehat{T}$. The set $\widehat{T}$ can be endowed with the metric 
\begin{equation*}
  \dhatt(x,y):=\left\{\begin{array}{cl} 2^{-|x\wedge y|} & \text{if}~x\neq y \\ 0 & \text{if}~x=y \end{array}\right.\,,
\end{equation*}
where $|x\wedge y|$ denotes the number of edges the two paths run together until they separate, see \circled{1} in Figure~\ref{fig:compactifications}. In other words, $|x\wedge y|$ is the maximal number $k\in\mathbb{N}_{0}$ such that $x$ and $y$ are both defined at $k$ and the vertices $x(k)$ and $y(k)$ agree. Hence, the later the paths separate the closer they are. The set $\widehat{T}$ endowed with the metric $\dhatt$ is a compact metric space that contains $G/B$ as a discrete and dense subset. The complement of $G/B$ is the set of infinite paths, it is denoted by $\partial T$ and called the \emph{space of ends}.

In case of $\mathbb{H}$, we temporarily switch to the Poincar\'e disc model. Instead of working in the half-plane $\mathbb{H}=\set{z\in\mathbb{C}}{\Im(z)>0}$, we consider the open unit disc $\mathbb{D}:=\set{z\in\mathbb{C}}{|z|<1}$. The Cayley transform $W:\mathbb{H}\hooktwoheadrightarrow\mathbb{D}$ given by $W(z):=\frac{z-i}{z+i}$ is one possibility to convert between the two models. The hyperbolic topology on $\mathbb{D}$ is the one induced by the Cayley transform. It agrees with the standard topology on~$\mathbb{D}$ so that, topologically speaking, the hyperbolic plane in the Poincar\'e disc model is just a subspace of the complex plane $\mathbb{C}$. We may therefore compactify it by taking the closed unit disc $\widehat{\mathbb{D}}:=\set{z\in\mathbb{C}}{|z|\leq 1}$, see Figure~\ref{fig:compactifications}. In order to translate this compactification back to the Poincar\'e half-plane model, we first extend both the domain and the codomain of the Cayley transform so that we obtain a bijection $W:\mathbb{H}\cup\mathbb{R}\cup\sset{\infty}\hooktwoheadrightarrow\widehat{\mathbb{D}}$, and then apply its inverse. The space $\widehat{\mathbb{H}}:=\mathbb{H}\cup\mathbb{R}\cup\sset{\infty}$ is our compactification. It is, once again, endowed with the induced topology, and thus a compact space that contains $\mathbb{H}$ as a dense subset. The complement of $\mathbb{H}$ is the union $\mathbb{R}\cup\sset{\infty}$, it is denoted by $\partial\mathbb{H}$ and called the \emph{hyperbolic boundary}. Having introduced the hyperbolic boundary this way, the following lemma gives us a helpful criterion for convergence. Its proof is elementary and we leave it to the reader.

\begin{lem}
\label{lem:convergence}
A sequence $(x_{0},x_{1},\dotsc)$ in $\mathbb{H}$ converges to $r\in\mathbb{R}=\partial\mathbb{H}\smallsetminus\sset{\infty}$ if and only if it does with respect to the standard topology on the complex plane $\mathbb{C}$. The sequence converges to $\infty\in\partial\mathbb{H}$ if and only if the absolute values $|x_{n}|$ tend to $\infty$.
\end{lem}

\begin{figure}
\begin{center}
  \includegraphics{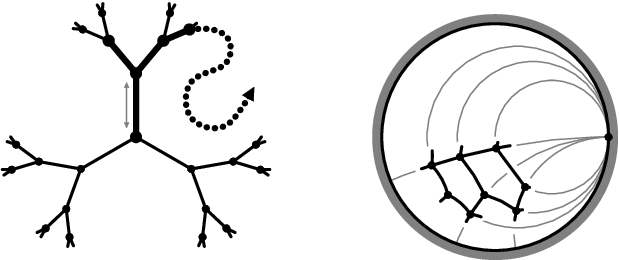}
  \begin{picture}(0,0)
    \put(-258,97){$x$}
    \put(-252,70.75){\circled{1}}
    \put(-196,72){$y$}
    \put(-240,45){$B$}
  \end{picture}
\end{center}
\caption{The space of ends (left) and the hyperbolic boundary in the Poincar\'e disc model (right).}
\label{fig:compactifications}
\end{figure}
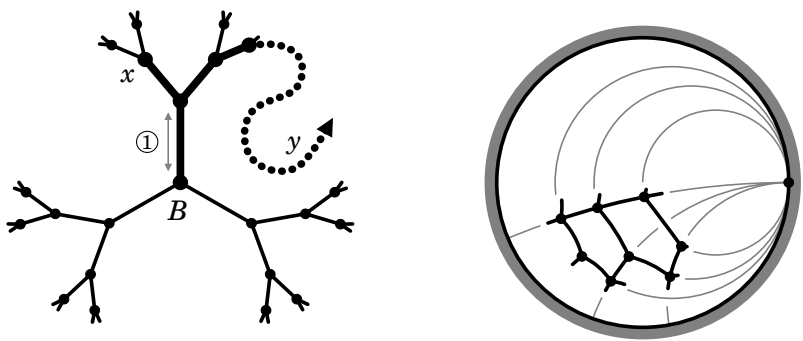


\section{Random walks on groups}
\label{sec:randomwalks}


\subsection{Preliminaries}
\label{sub:p}

The aim of the present paper is to study random walks. Given a countable state space $X$, an initial probability measure $\vartheta:X\to[0,1]$, and transition probabilities $p:X\times X\to[0,1]$, we are interested in the Markov chain $Z=(Z_{0},Z_{1},\dotsc)$ that starts according to $\vartheta$ and proceeds according to $p$. Formally, we construct the probability space $(\Omega,\mathcal{A},\mathbb{P})$, where $\Omega:=\set{(x_{0},x_{1},\dotsc)}{x_{n}\in X}$ is the set of trajectories, $\mathcal{A}$ is the product $\sigma$-algebra, and $\mathbb{P}$ is the probability measure induced by $\vartheta$ and~$p$. The projections $Z_{n}:\Omega\to X$ given by $Z_{n}(x_{0},x_{1},\dotsc):=x_{n}$ then become random variables that constitute the Markov chain. We shall use the term \emph{random walk} instead of Markov chain.

Assume now that $X$ is a countable group $G$. In this situation, we may adapt the transition probabilities $p:G\times G\to [0,1]$ to the group structure. More precisely, let $\mu:G\to [0,1]$ be a probability measure on $G$ and consider the random walk given by the following data. The initial probability measure $\vartheta:G\to[0,1]$ puts all mass on the identity element $1\in G$ and the transition probabilities $p:G\times G\to [0,1]$ are given by $p(g,h):=\mu(g^{-1}h)$. We could also have said $p(g,gx):=\mu(x)$, which leads to a handy interpretation. The random walk starts at the identity element and has independent $\mu$-distributed increments being multiplied from the right to the current state. Hence, $Z_{0}=1$ a.\,s.~(=\,almost surely) and for every $n\in\mathbb{N}$, $Z_{n}=X_{1}\cdot\dotsc\cdot X_{n}$, where $X_{1},X_{2},\dotsc$ is a sequence of independent $\mu$-distributed random variables.

\emph{Throughout the present paper, we assume that the support $\supp(\mu)=\set{g\in G}{\mu(g)>0}$ generates $G$ as a semigroup. In other words, the random walk is irreducible; any two states can be reached from each other with positive probability.}

Given a probability space, e.\,g.~$(\Omega,\mathcal{A},\mathbb{P})$ described above, and a real valued random variable $X:\Omega\rightarrow\mathbb{R}$, the latter has finite first moment if $\int|X|\d{\mathbb{P}}$ is finite. In this case, both $\int X^{+}\d{\mathbb{P}}$ and $\int X^{-}\d{\mathbb{P}}$ are finite and we can define the expectation $\mathbb{E}(X):=\int X^{+}\d{\mathbb{P}}-\int X^{-}\d{\mathbb{P}}$. Of course, the difference would still make sense if only one of the two integrals was finite. But this case is not of relevance for us and when writing $\mathbb{E}(X)$ we implicitly mean that $\mathbb{E}(X)$ is a real number. More generally, given any non-negative $k\in\mathbb{R}$, the random variable $X:\Omega\rightarrow\mathbb{R}$ has finite $k$-th moment if $\int|X|^{k}\d{\mathbb{P}}$ is finite.

\begin{defn}[``word metric'']
If $G$ is a finitely generated group and $S\subseteq G$ is a finite generating set, then the word metric $\dwordS:G\times G\to\mathbb{N}_{0}$ is the distance in the respective Cayley graph. In other words,
\begin{equation*}
  \dwordS(g,h):=\min\set{n\in\mathbb{N}_{0}}{\exists\,s_{1},\dotsc,s_{n}\in S~\text{and}~\varepsilon_{1},\dotsc,\varepsilon_{n}\in\sset{1,-1}~\text{such that}~g^{-1}h={s_{1}}^{\varepsilon_{1}}\cdot\dotsc\cdot {s_{n}}^{\varepsilon_{n}}}\,.
\end{equation*}
\end{defn}

\begin{defn}[``finite k-th moment'' for G-valued random variables]
Let $G$ be a finitely generated group and $S\subseteq G$ be a finite generating set. A random variable $X:\Omega\to G$ has finite $k$-th moment if the image $\dwordS(1,X):\Omega\to\mathbb{N}_{0}$ has finite $k$-th moment in the classical sense, i.\,e.~if $\int\dwordS(1,X)^{k}\d{\mathbb{P}}$ is finite. It is well-known that this property does not depend on the choice of $S$.
\end{defn}


\subsection{Lebesgue--Rohlin spaces}

In order to define the Poisson--Furstenberg boundary, we need to ensure that we are working with \emph{Lebesgue--Rohlin spaces}, which are also known as standard probability spaces. For definitions, basic examples, and fundamental results, we refer to \cite{rohlin-1952}, \cite{haezendonck-1973}, and \cite{rudolph-1990}. Two short summaries can also be found in \cite[Appendix 1]{cfs-1982} and \cite[Appendix]{kkr-2004}.

\begin{exa}
\label{exa:polish}
A Polish space is a topological space that is separable and completely metrisable. All Polish spaces endowed with their Borel $\sigma$-algebra~$\mathcal{B}$ and a Borel measure $\mu$ become, after completion, examples of Lebesgue--Rohlin spaces, see \cite[\S 2.7]{rohlin-1952} and \cite[p.\,248, Example 1]{haezendonck-1973}.
\end{exa}

In light of Example \ref{exa:polish}, we observe that the space of trajectories $\Omega$ introduced in Section \ref{sub:p} is the product $X^{\mathbb{N}_{0}}$ and can therefore be endowed with the product topology. The resulting space is actually Polish, see e.\,g.~\cite[Theorem 24.11]{willard-1970}. Since its Borel $\sigma$-algebra agrees with the product $\sigma$-algebra~$\mathcal{A}$, the completion of $(\Omega,\mathcal{A},\mathbb{P})$ is a Lebesgue--Rohlin space. Let us assume that, as soon as a measurable space is endowed with a probability measure, we are working with its completion. We may therefore say that $(\Omega,\mathcal{A},\mathbb{P})$ is a Lebesgue--Rohlin space.

\emph{As it is customary, we always deal with Lebesgue--Rohlin spaces mod 0, i.\,e.~up to subsets of measure 0, even when it is not indicated explicitly.}


\subsection{Poisson--Furstenberg boundary}

Several equivalent definitions of the Poisson--Furstenberg boundary are given in \cite{kv-1983}. Since we are interested in the long-time behaviour of the trajectories $x=(x_{0},x_{1},\dotsc)\in\Omega$, we identify those pairs of trajectories whose tails sooner or later behave identically. More precisely, we define an equivalence relation $\sim$ on $\Omega$ by setting
\begin{equation*}
x\sim y:\iff\exists\,t_{1},t_{2}\in\mathbb{N}_{0}~\text{such that}~\forall\,n\in\mathbb{N}_{0}~\text{the equation}~x_{t_{1}+n}=y_{t_{2}+n}~\text{holds.}
\end{equation*}
Consider the partition~$\zeta$ of $\Omega$ into the equivalence classes mod $\sim$. It induces a sub-$\sigma$-algebra $\mathcal{A}(\zeta)\subseteq\mathcal{A}$ consisting of all those measurable sets $A\in\mathcal{A}$ that are, mod 0, unions of elements in $\zeta$. Being a complete sub-$\sigma$-algebra, $\mathcal{A}(\zeta)$ corresponds to a measurable partition $\zeta_{1}$ of $\Omega$. This partition is unique, up to equivalence mod 0. It has the properties that the induced sub-$\sigma$-algebra $\mathcal{A}(\zeta_{1})\subseteq\mathcal{A}$ coincides with~$\mathcal{A}(\zeta)$ and that the quotient of $(\Omega,\mathcal{A},\mathbb{P})$ by $\zeta_{1}$ is a Lebesgue--Rohlin space.

The latter is called the \emph{Poisson--Furstenberg boundary} and we denote it by $(B,\mathcal{B},\nu)$. It is naturally endowed with a \emph{boundary map} $\bnd:\Omega\to B$ assigning to every trajectory $x\in\Omega$ the element in $\zeta_{1}$ that contains it. The boundary map is a measurable and measure-preserving map between Lebesgue--Rohlin spaces; such a map is called a \emph{homomorphism}.

Here, we consider irreducible random walks on countable groups $G$. In this situation, the measurable $G$-action on $\Omega$ given by $g(x_{0},x_{1},\dotsc):=(gx_{0},gx_{1},\dotsc)$ induces a measurable $G$-action on $B$ so that the measure $\nu$ is \emph{$\mu$-stationary}, i.\,e.~$\nu(A)=\sum_{g\in G}\mu(g)\cdot\nu(g^{-1}A)$, and \emph{quasi-invariant}, i.\,e.~the $G$-action maps null sets $A$ to null sets $gA$.


\subsection{Some results about the Poisson--Furstenberg boundary}
\label{sub:sratpfb}

Given such a random walk, it is a challenging problem to decide whether the Poisson--Furstenberg boundary is trivial or not. In the latter case, one may wonder how to identify it geometrically. We shall only address a few results; a survey was given by Erschler in \cite{erschler-2010}. As always, we assume that the random walk is irreducible.

If $G$ is abelian, then the Poisson--Furstenberg boundary is trivial, see \cite{blackwell-1955} and \cite{cd-1960}. The same holds true for all groups of polynomial growth, and for groups of subexponential growth endowed with a probability measure $\mu$ with finite first moment. For the special case of probability measures with finite support, see \cite{avez-1974}, and for the general case, see e.\,g.~\cite[Theorem 5.3]{kw-2002} and \cite[\S 4]{erschler-2004}. Moreover, it was shown in \cite{erschler-2004}, that the assumption of finite first moment cannot be dropped. If $G$ is amenable, then there is at least one symmetric probability measure $\mu$ such that the Poisson--Furstenberg boundary is trivial, see the conjecture in \cite[\S 9]{furstenberg-1973}. The proof was announced in \cite[Theorem 4]{vk-1979} and given in \cite{rosenblatt-1981} and \cite{kv-1983}.

For random walks on the Baumslag--Solitar group $G=\BS(1,2)$ with finite first moment, one can be more specific. Here, the Poisson--Furstenberg boundary is isomorphic to $\mathbb{R}$ for $\delta<0$, trivial for $\delta=0$, and isomorphic to $\mathbb{Q}_{2}$ for $\delta>0$, see \cite[Theorem~5.1]{kaimanovich-1991}. Further results about random walks on rational affinities are given in \cite{brofferio-2006}. If $G$ is non-amenable, then the Poisson--Furstenberg boundary can never be trivial, see \cite[\S 9]{furstenberg-1973} and \cite[\S4.2]{kv-1983}. This holds in particular for random walks on non-amenable Baumslag--Solitar groups, even when $\delta=0$.

\begin{rem}
There are striking similarities between solvable Baumslag--Solitar groups and lamplighter groups. For example, while $\BS(1,2)$ can be described as the semidirect product $\mathbb{Z}\big[\frac{1}{2}\big]\rtimes\mathbb{Z}$, where $1\in\mathbb{Z}$ acts by doubling, the lamplighter group $\mathbb{Z}_{2}\wr\mathbb{Z}$ is defined as $\big(\bigoplus_{i\in\mathbb{Z}}\mathbb{Z}_{2}\big)\rtimes\mathbb{Z}$, where $1\in\mathbb{Z}$ acts by shifting the index by 1. For results on the Poisson--Furstenberg boundary of random walks on lamplighter groups, see \cite{vk-1979}, \cite{kv-1983}, \cite{lp-2015}, and also \cite{sava-2010}.
\end{rem}


\subsection{Kaimanovich's strip criterion}

Kaimanovich's strip criterion, which we recall below, is a tool for identifying the Poisson--Furstenberg boundary geometrically. We state it as a theorem and then briefly discuss the notions appearing in the statement. For the proof, we refer to \cite[\S 6.4]{kaimanovich-2000}.

\begin{thm}[``strip criterion'']
\label{thm:strip}
Let $Z=(Z_{0},Z_{1},\dotsc)$ be a random walk on a countable group $G$ driven by a probability measure $\mu$ with finite entropy $H(\mu)$. Moreover, let $(B_{-},\mathcal{B}_{-},\nu_{-})$ and $(B_{+},\mathcal{B}_{+},\nu_{+})$ be $\check{\mu}$- and $\mu$-boundaries, respectively. If there exist a gauge $\mathcal{G}=(\mathcal{G}_{1},\mathcal{G}_{2},\dotsc)$ on $G$ with associated gauge function $|\cdot|=|\cdot|_{\mathcal{G}}$ and a measurable $G$-equivariant map $\mathcal{S}$ assigning to pairs of points $(b_{-},b_{+})\in B_{-}\times B_{+}$ non-empty strips $\mathcal{S}(b_{-},b_{+})\subseteq G$ such that for every $g\in G$ and $\nu_{-}\otimes\nu_{+}$-almost every $(b_{-},b_{+})\in B_{-}\times B_{+}$ 
\begin{equation*}
  \frac{1}{n}\cdot\ln\left(\card\left(\mathcal{S}(b_{-},b_{+})g\cap\mathcal{G}_{|Z_{n}|}\right)\right)\,\xrightarrow{~n\to\infty~}\,0~~\text{in probability},
\end{equation*}
then the $\mu$-boundary $(B_{+},\mathcal{B}_{+},\nu_{+})$ is maximal.
\end{thm}

\begin{rem}
\label{rem:removeg}
The proof shows that it is not even necessary to verify the convergence for every $g\in G$. It suffices to consider the special case $g=1$ as long as we can ensure that a random strip contains the identity element $1\in G$ with positive probability, i.\,e.~that $\nu_{-}\otimes\nu_{+}\set{(b_{-},b_{+})\in B_{-}\times B_{+}}{1\in \mathcal{S}(b_{-},b_{+})}>0$.
\end{rem}

The \emph{entropy} of the probability measure $\mu$ is given by $H(\mu):=\sum_{g\in G}-\log_{2}(\mu(g))\cdot\mu(g)$. Here, as usual, one defines $-\log_{2}(0)\cdot 0:=0$. The assumption of finite entropy will be no issue for us because Baumslag--Solitar groups are finitely generated and the increments of the random walks under consideration will all have finite first moment. This implies that their probability measures~$\mu$ have finite entropy, as stated in the following lemma. Its proof is elementary. For an idea, see e.\,g.~\cite[Theorem 4.1]{gps-1994}.

\begin{lem}
\label{lem:entropy}
Let $G$ be a finitely generated group and let $\mu:G\to[0,1]$ be a probability measure. If a $\mu$-distributed random variable $X:\Omega\to G$ has finite first moment, then $\mu$ has finite entropy.
\end{lem}

Kaimanovich defines a \emph{$\mu$-boundary} as the quotient of the Poisson--Furstenberg boundary with respect to some $G$-invariant measurable partition, see e.\,g.~\cite[\S 1.5]{kaimanovich-2000}. The Poisson--Furstenberg boundary is therefore itself a $\mu$-boundary, the \emph{maximal} one. Moreover, every Lebesgue--Rohlin space $(B_{+},\mathcal{B}_{+},\nu_{+})$ endowed with a measurable $G$-action and a homomorphism $\bnd_{+}:\Omega\to B_{+}$ that is \circled{1} $\sim$-invariant and \circled{2}~$G$-equivariant is a $\mu$-boundary. At this point, recall that the random walk is irreducible, whence the properties \circled{1} and \circled{2} already imply that the measure $\nu_{+}$ is $\mu$-stationary and quasi-invariant.

While~$\mu$ is the probability measure that drives the random walk, the symbol $\check{\mu}$ denotes the reflected probability measure, which is given by $\check{\mu}(g):=\mu(g^{-1})$. Accordingly, a \emph{$\check{\mu}$-boundary} is a Lebesgue--Rohlin space $(B_{-},\mathcal{B}_{-},\nu_{-})$ that satisfies the requirements of a $\mu$-boundary when replacing $\mu$ by $\check{\mu}$. 

A \emph{gauge} $\mathcal{G}$ is an exhaustion $\mathcal{G}=(\mathcal{G}_{1},\mathcal{G}_{2},\dotsc)$ of the group $G$, i.\,e.~a sequence of subsets $\mathcal{G}_{k}\subseteq G$ which is increasing $\mathcal{G}_{1}\subseteq\mathcal{G}_{2}\subseteq\dotsc$ and whose union $\mathcal{G}_{1}\cup\mathcal{G}_{2}\cup\dotsc$ is the whole group $G$. Given a gauge $\mathcal{G}$ and an element $g\in G$, we may ask for the minimal index $k\in\mathbb{N}$ with the property that $g\in\mathcal{G}_{k}$. This index is the value of the associated gauge function $|\cdot|=|\cdot|_{\mathcal{G}}$ at $g$.

\begin{rem}
Kaimanovich distinguishes between various kinds of gauges, see \cite{kaimanovich-2000}. For example, a gauge $\mathcal{G}$ is subadditive if any two group elements $g_{1},g_{2}\in G$ satisfy $|g_{1}g_{2}|\leq|g_{1}|+|g_{2}|$ and it is temperate if all gauge sets $\mathcal{G}_{k}$ are finite and grow at most exponentially. Even though these two properties do play a crucial role in the corollaries to the strip criterion given in \cite[\S 6.5]{kaimanovich-2000}, they are not required in the strip criterion itself. And, in fact, not all of our gauges will have these two properties.
\end{rem}

The power set $\sset{0,1}^{G}$ is naturally endowed with the product $\sigma$-algebra, which enables us to talk about \emph{measurability} of the map $\mathcal{S}:B_{-}\times B_{+}\rightarrow\sset{0,1}^{G}$. More precisely, the product $\sigma$-algebra on $\sset{0,1}^{G}$ is generated by the coordinate projections. Since the set $\sset{0,1}$ consists of only two elements, the $\sigma$-algebra is already generated by the preimages of $1\in\sset{0,1}$. In order to show that $\mathcal{S}$ is measurable, it thus suffices to verify that for every $g\in G$ the set of all $(b_{-},b_{+})\in B_{-}\times B_{+}$ whose strip $\mathcal{S}(b_{-},b_{+})\subseteq G$ contains the element $g\in G$ is measurable. As soon as we know that $\mathcal{S}$ is $G$-equivariant, it even suffices to verify measurability for $g=1$, which will be immediate for the strips under consideration. 


\section{Identification of the Poisson--Furstenberg boundary}
\label{sec:identification}


\subsection{Convergence to the boundary of the hyperbolic plane}
\label{sub:cttbothp}

Let us now return to $G=\BS(p,q)$ with $1\leq p<q$ and consider a random walk $Z=(Z_{0},Z_{1},\dotsc)$ on $G$. When working with the projection $\pi_{\mathbb{H}}:G\rightarrow\mathbb{H}$, we may analyse the imaginary part $\Im(\pi_{\mathbb{H}}(g))$ and the real part $\Re(\pi_{\mathbb{H}}(g))$ separately, and it is convenient to abbreviate the former by $A_{g}$ and the latter by~$B_{g}$. Occasionally, we do not assume that $X_{1}$ has some finite moment but impose this assumption on the images $\ln(A_{X_{1}})$ and $\ln(1+|B_{X_{1}}|)$. The following lemma relates the two situations.

\begin{lem}
\label{lem:assumptions}
If $X_{1}$ has finite $k$-th moment, then $\ln(A_{X_{1}})$ and $\ln(1+|B_{X_{1}}|)$ have finite $k$-th moment, too.
\end{lem}

\begin{rem}
\label{rem:multiple}
It follows from the definition of $\pi_{\mathbb{H}}$ that for every $g\in G$ the equation $A_{g}=q^{\lambda(g)}p^{-\lambda(g)}$ holds. Taking the logarithm on both sides yields the equation $\ln(A_{g})=\ln\big(\frac{q}{p}\big)\cdot\lambda(g)$. So, instead of thinking of\,\,$\ln(A_{g})$ we may think of a multiple of $\lambda(g)$.
\end{rem}

\vspace{-10pt}\begin{proof}[Proof of Lemma \ref{lem:assumptions}]
Let $S:=\sset{a,b}\subseteq G$ be the standard generating set. Then,
\begin{equation*}
  \int|\ln(A_{X_{1}})|^{k}\d{\mathbb{P}}=\left(\ln\left(\frac{q}{p}\right)\right)^{k}\cdot\int|\lambda(X_{1})|^{k}\d{\mathbb{P}}\leq\left(\ln\left(\frac{q}{p}\right)\right)^{k}\cdot\underbrace{\int\dwordS(1,X_{1})^{k}\d{\mathbb{P}}}_{<\;\infty}<\infty\,.
\end{equation*}
This proves the first assertion. For the second one, Lemma~\ref{lem:lookingfromfront} implies that the distance $\dhyp(\pi_{\mathbb{H}}(1),\pi_{\mathbb{H}}(g))$ is at most $\max\sset{\ell_{a},\ell_{b}}\cdot\dwordS(1,g)$. This allows us to estimate $\ln(1+|B_{g}|)$ by a multiple of $\dwordS(1,g)$. Indeed, 
\begin{align*}
  \ln(1+|B_{g}|) &\leq\ln\left(1+\frac{1}{2}\cdot|B_{g}|^2+\sqrt{\left(1+\frac{1}{2}\cdot|B_{g}|^2\right)^{2}-1}\;\right)=\arcosh\left(1+\frac{1}{2}\cdot|B_{g}|^2\right) \\ &=\dhyp(i,i+B_{g})\leq\dhyp(i,A_{g}\cdot i+B_{g})+\dhyp(A_{g}\cdot i+B_{g},i+B_{g})=\dhyp(\pi_{\mathbb{H}}(1),\pi_{\mathbb{H}}(g))+|\ln(A_{g})| \\[6pt] &\leq\max\sset{\ell_{a},\ell_{b}}\cdot\dwordS(1,g)+\ln\left(\frac{q}{p}\right)\cdot|\lambda(g)|\leq\max\sset{\ell_{a},\ell_{b}}\cdot\dwordS(1,g)+\ln\left(\frac{q}{p}\right)\cdot\dwordS(1,g)\,.
\end{align*}
Therefore,
\begin{equation*}
  \int\ln(1+|B_{X_{1}}|)^{k}\d{\mathbb{P}}\leq\left(\max\sset{\ell_{a},\ell_{b}}+\ln\left(\frac{q}{p}\right)\right)^{k}\cdot\underbrace{\int\dwordS(1,X_{1})^{k}\d{\mathbb{P}}}_{<\;\infty}<\infty\,,
\end{equation*}
which proves the second assertion.
\end{proof}

\begin{defn}[``vertical drift'']
If\,\,$\ln(A_{X_{1}})$ has finite first moment, then $\lambda(X_{1})$ has finite first moment and we can define the expectation $\mathbb{E}(\lambda(X_{1}))$. We call the latter the vertical drift and denote it by~$\delta$.
\end{defn}

The following lemmas concern the behaviour of the projections $\pi_{\mathbb{H}}(Z_{n})$. They seem to be well-known and we do not claim originality. But, for the sake of completeness, we give rigorous proofs.

\begin{lem}
\label{lem:positivedrifthyp}
Assume that $\ln(A_{X_{1}})$ has finite first moment. If $Z$ has positive vertical drift $\delta>0$, then the projections $\pi_{\mathbb{H}}(Z_{n})$ converge a.\,s.~to $\infty\in\partial\mathbb{H}$.
\end{lem}

\vspace{-10pt}\begin{proof}
By the strong law of large numbers,
\begin{equation*}
  \frac{\lambda(Z_{n})}{n}=\frac{\lambda(X_{1})+\dotsc+\lambda(X_{n})}{n}\,\xrightarrow[\text{\raisebox{.5ex}{a.\,s.}}]{~n\to\infty~}\,\mathbb{E}(\lambda(X_{1}))=\delta>0\,.
\end{equation*}
Therefore, the projections $\lambda(Z_{n})$ tend a.\,s.~to $\infty$. By Remark \ref{rem:multiple}, so do the imaginary parts~$A_{Z_{n}}$ and the absolute values $|\pi_{\mathbb{H}}(Z_{n})|$. Now, Lemma~\ref{lem:convergence} completes the proof.
\end{proof}

\begin{lem}
\label{lem:negativedrifthyp}
Assume that both $\ln(A_{X_{1}})$ and $\ln(1+|B_{X_{1}}|)$ have finite first moment. If $Z$ has negative vertical drift $\delta<0$, then the projections $\pi_{\mathbb{H}}(Z_{n})$ converge a.\,s.~to a random element $r\in\mathbb{R}=\partial\mathbb{H}\smallsetminus\sset{\infty}$.
\end{lem}

\vspace{-10pt}\begin{proof}
The proof of Lemma \ref{lem:positivedrifthyp} can be adapted to show that the imaginary parts~$A_{Z_{n}}$ converge a.\,s.~to~$0$, whence we only need to understand the behaviour of the real parts $B_{Z_{n}}$. By the construction of the group homomorphism $\pi_{\isom(\mathbb{H})}:G\to\isom(\mathbb{H})$, each isometry $\pi_{\isom(\mathbb{H})}(g)$ with $g\in G$ is of the form $z\mapsto\alpha z+\beta$ with $\alpha,\beta\in\mathbb{R}$ and $\alpha>0$. So, the equation $\pi_{\mathbb{H}}(g)=A_{g}\cdot i+B_{g}$ yields $\pi_{\isom(\mathbb{H})}(g)(z)=A_{g}\cdot z+B_{g}$ and, in light of the multiplication $(\pi_{\isom(\mathbb{H})}(g_{1})\circ\pi_{\isom(\mathbb{H})}(g_{2}))(z)=A_{g_{1}}\cdot A_{g_{2}}\cdot z+A_{g_{1}}\cdot B_{g_{2}}+B_{g_{1}}$, we obtain
\begin{align*}
  \pi_{\mathbb{H}}(Z_{n}) &=\pi_{\isom(\mathbb{H})}(Z_{n})(i)=\pi_{\isom(\mathbb{H})}(X_{1}\cdot\dotsc\cdot X_{n})(i) \\[5pt] &=(\pi_{\isom(\mathbb{H})}(X_{1})\circ\dotsc\circ\pi_{\isom(\mathbb{H})}(X_{n}))(i) \\ &=A_{X_{1}}\cdot\dotsc\cdot A_{X_{n}}\cdot i+\sum_{k=1}^{n}A_{X_{1}}\cdot\dotsc\cdot A_{X_{k-1}}\cdot B_{X_{k}}\,.
\end{align*}
Therefore, the real parts $B_{Z_{n}}$ are the partial sums of the series $\sum_{k=1}^{\infty}C_{k}$ with $C_{k}:=A_{X_{1}}\cdot\dotsc\cdot A_{X_{k-1}}\cdot B_{X_{k}}$. In order to verify a.\,s.~convergence of this series, we apply Cauchy's root test, 
\begin{equation*}
  |C_{k}|^{\frac{1}{k}}\leq\exp\left(\ln\left(\frac{q}{p}\right)\cdot\underbrace{\frac{\lambda(X_{1})+\dotsc+\lambda(X_{k-1})}{k-1}}_{\to\;\mathbb{E}(\lambda(X_{1}))\;=\;\delta\;<\;0~\text{a.\,s.}}\,\cdot\,\underbrace{\frac{k-1}{k}}_{\to\;1}\,\right)\cdot\exp\left(\underbrace{\frac{\ln(1+|B_{X_{k}}|)}{k}}_{\to\;0~\text{a.\,s.}}\,\right)\,\xrightarrow[\text{\raisebox{.5ex}{a.\,s.}}]{~k\to\infty~}\,\left(\frac{q}{p}\right)^{\delta}<1\,.
\end{equation*}
The convergence claimed in the first factor follows from the strong law of large numbers, the one claimed in the second factor from the Borel--Cantelli lemma. Indeed, let $Q_{k}:=\frac{1}{k}\cdot\ln(1+|B_{X_{k}}|)$. In order to show that $Q_{k}$ converges a.\,s.~to $0$, recall that $\ln(1+|B_{X_{1}}|)$ has finite first moment. For every $\varepsilon>0$ we may thus estimate
\begin{equation*}
  \sum_{k=1}^{\infty}\mathbb{P}(Q_{k}>\varepsilon)\leq\sum_{k=1}^{\infty}\mathbb{P}\left(\left\lceil\,\frac{\ln(1+|B_{X_{1}}|)}{\varepsilon}\,\right\rceil\geq k\right)=\mathbb{E}\left(\left\lceil\,\frac{\ln(1+|B_{X_{1}}|)}{\varepsilon}\,\right\rceil\right)\,.
\end{equation*}
Now, the Borel--Cantelli lemma yields $\mathbb{P}(\exists\;\text{infinitely many}~k\in\mathbb{N}~\text{such that}~Q_{k}>\varepsilon)=0$, from where we may conclude that $Q_{k}$ converges a.\,s.~to $0$, as claimed above. Therefore, $\limsup_{k\to\infty}|C_{k}|^{\frac{1}{k}}<1$ a.\,s., whence $\sum_{k=1}^{\infty}C_{k}$ converges a.\,s.~to a random element $r\in\mathbb{R}$.
\end{proof}

What remains is the driftless case. An answer was given by Brofferio in \cite[Theorem 1]{brofferio-2003}. It says that if $\ln(A_{X_{1}})$ and $\ln(1+|B_{X_{1}}|)$ have finite first moment, then the projections $\pi_{\mathbb{H}}(Z_{n})$ converge a.\,s.~to $\infty\in\partial\mathbb{H}$. For us, a result of slightly different flavour will be of relevance.

\begin{lem}
\label{lem:zerodrifthyp}
Assume that $\ln(A_{X_{1}})$ has finite second moment and there is an $\varepsilon>0$ such that $\ln(1+|B_{X_{1}}|)$ has finite $(2+\varepsilon)$-th moment. If $Z$ has no vertical drift, i.\,e.~$\delta=0$, then the projections $\pi_{\mathbb{H}}(Z_{n})$ have sublinear speed, i.\,e.
\begin{equation*}
  \frac{\dhyp(\pi_{\mathbb{H}}(Z_{0}),\pi_{\mathbb{H}}(Z_{n}))}{n}\,\xrightarrow[\text{\raisebox{.5ex}{a.\,s.}}]{~n\to\infty~}\,0\,.
\end{equation*}
\end{lem}

The proof is based on ideas that go back to \'Elie \cite[Lemme 5.49]{elie-1982} and have also been used in \cite[Proposition 4b]{ckw-1994}. We first adapt these ideas to our situation in Lemma \ref{lem:elie} and then deduce Lemma~\ref{lem:zerodrifthyp}. By assumption, there is no vertical drift so that the pointwise projection $\lambda(Z)=(\lambda(Z_{0}),\lambda(Z_{1}),\dotsc)$ to $\mathbb{Z}$ is recurrent. In particular, we know that there exists a.\,s.~a strictly increasing sequence $\tau(0),\tau(1),\dotsc$ given by $\tau(0):=0$ and $\tau(n):=\inf\,\set{k\in\mathbb{N}}{k>\tau(n-1)~\text{and}~\lambda(Z_{k})>\lambda(Z_{\tau(n-1)})}$ for all $n\in\mathbb{N}$. We call $\tau(n)$ the \emph{$n$-th ladder time}, see Figure \ref{fig:upwardstopping} for an illustration, and write $\tau:=\tau(1)$ for short.

\begin{figure}
\begin{center}
  \includegraphics{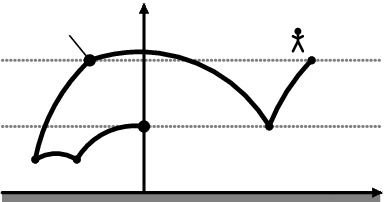}
  \begin{picture}(0,0)
    \put(-116,26){$\pi_{\mathbb{H}}(Z_{\tau(0)})$}
    \put(-184,86){$\pi_{\mathbb{H}}(Z_{\tau(1)})$}
  \end{picture}
\end{center}
\caption{The first ladder times $\tau(0)$ and $\tau(1)$.}
\label{fig:upwardstopping}
\end{figure}
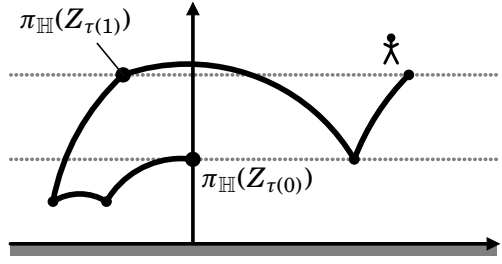

\begin{lem}
\label{lem:elie}
Under the same assumptions as in Lemma \ref{lem:zerodrifthyp}, the random variable $\ln(1+\sum_{k=1}^{\tau}|B_{X_{k}}|)$ has finite first moment.
\end{lem}

\vspace{-10pt}\begin{proof}
We adapt the proof of \cite[Lemme 5.49]{elie-1982} to our situation. Pick an $\varepsilon>0$ that satisfies the requirements of Lemmas \ref{lem:zerodrifthyp} and \ref{lem:elie} and let $\beta:=\frac{1}{2+\varepsilon}$. Since $\ln(A_{X_{1}})$ has finite second moment, we know that $\lambda(X_{1})$ has finite second moment and therefore $\mathbb{P}(\tau>k)$ is asymptotically equivalent to $\text{const}\cdot k^{-\frac{1}{2}}$ with a strictly positive constant, see \cite[\S5.44]{elie-1982} referring to \cite[p.\,415]{feller-1971}. In other words, the quotient of $\mathbb{P}(\tau>k)$ and $\text{const}\cdot k^{-\frac{1}{2}}$ converges to $1$. Thus,
\begin{equation*}
  \int\tau^{\beta}\d{\mathbb{P}}\leq\int\left\lceil\tau^{\beta}\right\rceil\d{\mathbb{P}}=\sum_{k=1}^{\infty}\mathbb{P}\left(\left\lceil\tau^{\beta}\right\rceil\geq k\right)=\sum_{k=0}^{\infty}\underbrace{\mathbb{P}\left(\tau>k^{\frac{1}{\beta}}\right)}_{\sim\;\text{const}\,\cdot\,k^{-\left(1+\frac{\varepsilon}{2}\right)}}.
\end{equation*}
In particular, there is a $k_{0}\in\mathbb{N}$ such that for all $k\geq k_{0}$ the inequality $\mathbb{P}\big(\tau>k^{\frac{1}{\beta}}\big)<k^{-\left(1+\frac{\varepsilon}{4}\right)}$ holds. Since $\sum_{k=k_{0}}^{\infty}k^{-\left(1+\frac{\varepsilon}{4}\right)}$ is finite, we know that $\int\tau^{\beta}\d{\mathbb{P}}$ is finite. By construction, the increments $\tau(1)-\tau(0)$, $\tau(2)-\tau(1),\dotsc$ are i.\,i.\,d.~(=\,independent and identically distributed), whence the fact that $0<\beta<1$, which implies that $(x+y)^{\beta}\leq x^{\beta}+y^{\beta}$, and the strong law of large numbers yield 
\begin{gather*}
  \frac{{\tau(n)}^{\beta}}{n}\leq\frac{(\tau(1)-\tau(0))^{\beta}+\dotsc+(\tau(n)-\tau(n-1))^{\beta}}{n}\,\xrightarrow[\text{\raisebox{.5ex}{a.\,s.}}]{~n\to\infty~}\,\mathbb{E}\left(\tau^{\beta}\right)\,,\implies\limsup_{n\to\infty}\frac{{\tau(n)}^{\beta}}{n}<\infty~~\text{a.\,s.} \label{eqn:preliminary} \tag{$\ast$}
\end{gather*}
Now, we are prepared for the main argument. The sums $\sum_{k=\tau(0)+1}^{\tau(1)}|B_{X_{k}}|,\;\sum_{k=\tau(1)+1}^{\tau(2)}|B_{X_{k}}|,\;\dotsc$ are i.\,i.\,d., they are non-negative and not a.\,s.~equal to zero. Hence, in view of \cite[Lemme 5.23]{elie-1982}, the following equivalence holds:
\begin{equation*}
  \int\ln\left(1+\sum_{k=1}^{\tau}|B_{X_{k}}|\right)\d{\mathbb{P}}<\infty\iff\underbrace{\limsup_{n\to\infty}\left(\sum_{k=\tau(n-1)+1}^{\tau(n)}|B_{X_{k}}|\right)^{\frac{1}{n}}}_{=:\;K}<\infty~~\text{a.\,s.}
\end{equation*}
It thus suffices to verify the right-hand side. In order to do so, we would like to estimate
\begin{equation*}
  K\leq\limsup_{n\to\infty}\exp\left(\frac{\ln\left(1+\sum_{k=1}^{\tau(n)}|B_{X_{k}}|\right)}{n}\right)\leq\exp\left(\underbrace{\limsup_{n\to\infty}\frac{\ln\left(1+\sum_{k=1}^{\tau(n)}|B_{X_{k}}|\right)}{{\tau(n)}^{\beta}}}_{=:\;L}\,\cdot\,\underbrace{\vphantom{\frac{\ln\left(1+\sum_{k=1}^{\tau(n)}|B_{X_{k}}|\right)}{{\tau(n)}^{\beta}}}\limsup_{n\to\infty}\frac{{\tau(n)}^{\beta}}{n}}_{<\,\infty~\text{a.\,s.~(\ref{eqn:preliminary})}}\right)\,.
\end{equation*}
A priori, it might be the case that $L$ is infinite and the second factor in the rightmost term is $0$, in which case the product would not make any sense. We claim that $L$ is a.\,s.~finite, which does not only legitimate the above estimate but also completes the proof. Indeed, observe that
\begin{align*}
  L &\leq\limsup_{n\to\infty}\frac{\ln\left(1+\tau(n)\cdot\max\set{|B_{X_{k}}|}{1\leq k\leq\tau(n)}\right)}{{\tau(n)}^{\beta}} \\[10pt] &\leq\underbrace{\limsup_{n\to\infty}\frac{\ln\left(\tau(n)\right)}{{\tau(n)}^{\beta}}}_{=\;0}\,+\limsup_{n\to\infty}\frac{\ln\left(1+\max\set{|B_{X_{k}}|}{1\leq k\leq\tau(n)}\right)}{{\tau(n)}^{\beta}} \\ &=\limsup_{n\to\infty}\left(\frac{\max\left\{\,\ln\left(1+|B_{X_{k}}|\right)^{\frac{1}{\beta}}\,:\,1\leq k\leq\tau(n)\,\right\}}{\tau(n)}\right)^{\beta}\leq\limsup_{n\to\infty}\left(\underbrace{\frac{\sum_{k=1}^{\tau(n)}\ln\left(1+|B_{X_{k}}|\right)^{\frac{1}{\beta}}}{\tau(n)}}_{=:\;M_{n}}\right)^{\beta}\,.
\end{align*}
Now, recall that $\frac{1}{\beta}=2+\varepsilon$. By the strong law of large numbers, $M_{n}$ converges a.\,s.~to $\mathbb{E}\big(\ln(1+|B_{X_{1}}|)^{\frac{1}{\beta}}\big)$. This implies that $\limsup_{n\to\infty}{M_{n}}^{\beta}$ is a.\,s.~finite, and so is $L$.
\end{proof}

\vspace{-10pt}\begin{proof}[Proof of Lemma~\ref{lem:zerodrifthyp}]
Recall from the proof of Lemma \ref{lem:elie}, that $\mathbb{P}(\tau>k)$ is asymptotically equivalent to $\text{const}\cdot k^{-\frac{1}{2}}$ with a strictly positive constant. In particular, there is a $k_{0}\in\mathbb{N}$ such that for all $k\geq k_{0}$ the inequality $\mathbb{P}(\tau>k)>k^{-1}$ holds, whence
\begin{equation*}
  \int\tau\d{\mathbb{P}}=\sum_{k=1}^{\infty}\mathbb{P}(\tau\geq k)=\sum_{k=0}^{\infty}\mathbb{P}(\tau>k)\geq\sum_{k=k_{0}}^{\infty}k^{-1}=\infty\,.
\end{equation*}
Since $\tau(1)-\tau(0),\tau(2)-\tau(1),\dotsc$ are i.\,i.\,d.~and non-negative, we may deduce from the strong law of large numbers by truncating the random variables, see e.\,g.~\cite[p.\,309, Lemma 6]{roussas-2014}, that
\begin{equation*}
  \label{eqn:truncating}
  \frac{\tau(n)}{n}=\frac{(\tau(1)-\tau(0))+(\tau(2)-\tau(1))+\dotsc+(\tau(n)-\tau(n-1))}{n}\,\xrightarrow[\text{\raisebox{.5ex}{a.\,s.}}]{~n\to\infty~}\,\infty\quad\text{and}\quad\frac{n}{\tau(n)}\,\xrightarrow[\text{\raisebox{.5ex}{a.\,s.}}]{~n\to\infty~}\,0\,. \tag{$\ast$}
\end{equation*}
This convergence can be used to estimate the distance between $\pi_{\mathbb{H}}(Z_{0})$ and $\pi_{\mathbb{H}}(Z_{n})$ from above. Indeed, for every $n\in\mathbb{N}_{0}$ let $m=m(n)\in\mathbb{N}_{0}$ be the unique element with $\tau(m)\leq n<\tau(m+1)$. It exists a.\,s.~because the ladder times $0=\tau(0)<\tau(1)<\dotsc$ do. Now, observe that
\begin{equation*}
  \frac{\dhyp(\pi_{\mathbb{H}}(Z_{0}),\pi_{\mathbb{H}}(Z_{n}))}{n}\leq\underbrace{\frac{\dhyp(i,A_{Z_{\tau(m)}}\cdot i)}{n}}_{\circled{1}}+\underbrace{\frac{\dhyp(A_{Z_{\tau(m)}}\cdot i,A_{Z_{\tau(m)}}\cdot i+B_{Z_{n}})}{n}}_{\circled{2}}+\underbrace{\frac{\dhyp(A_{Z_{\tau(m)}}\cdot i+B_{Z_{n}},A_{Z_{n}}\cdot i+B_{Z_{n}})}{n}}_{\circled{3}}\,.
\end{equation*}
The meaning of the three summands is illustrated Figure \ref{fig:distancedecomposition}. We will consider them separately and show that each of them converges a.\,s.~to $0$. For \circled{1} and \circled{3}, this is straightforward. Indeed, 
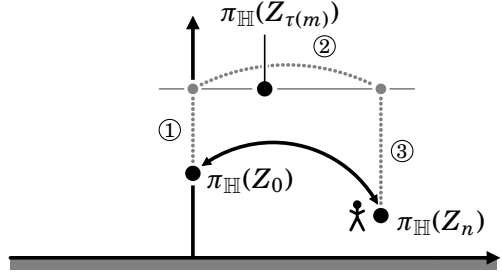
\begin{figure}
\begin{center}
  \includegraphics{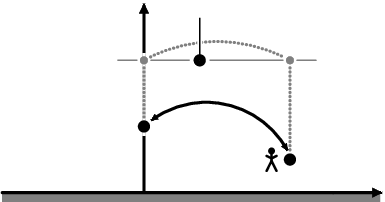}
  \begin{picture}(0,0)
    \put(-132,49){\circled{1}}
    \put(-74,80.5){\circled{2}}
    \put(-45,41){\circled{3}}
    \put(-109,94){$\pi_{\mathbb{H}}(Z_{\tau(m)})$}
    \put(-43,15){$\pi_{\mathbb{H}}(Z_{n})$}
    \put(-114,31){$\pi_{\mathbb{H}}(Z_{0})$}
  \end{picture}
\end{center}
\caption{Estimation of the distance between $\pi_{\mathbb{H}}(Z_{0})$ and $\pi_{\mathbb{H}}(Z_{n})$.}
\label{fig:distancedecomposition}
\end{figure}
\begin{equation*}
  \circled{1}=\frac{|\ln(A_{Z_{\tau(m)}})|}{n}\leq\frac{|\ln(A_{Z_{\tau(m)}})|}{\tau(m)}=\ln\left(\frac{q}{p}\right)\cdot\left|\frac{\lambda(X_{1})+\dotsc+\lambda(X_{\tau(m)})}{\tau(m)}\right|\,\xrightarrow[\text{\raisebox{.5ex}{a.\,s.}}]{~n\to\infty~}\,\ln\left(\frac{q}{p}\right)\cdot|\mathbb{E}(\lambda(X_{1}))|=\ln\left(\frac{q}{p}\right)\cdot|\delta|=0
\end{equation*}
and similarly
\begin{equation*}
  \circled{3}\leq\frac{\dhyp(A_{Z_{\tau(m)}}\cdot i,i)}{n}+\frac{\dhyp(i,A_{Z_{n}}\cdot i)}{n}=\circled{1}+\frac{|\ln(A_{Z_{n}})|}{n}=\circled{1}+\ln\left(\frac{q}{p}\right)\cdot\left|\frac{\lambda(X_{1})+\dotsc+\lambda(X_{n})}{n}\right|\,\xrightarrow[\text{\raisebox{.5ex}{a.\,s.}}]{~n\to\infty~}\,0\,.
\end{equation*}
For \circled{2}, recall from the proof of Lemma \ref{lem:negativedrifthyp} that $B_{Z_{n}}=\sum_{k=1}^{n}A_{X_{1}}\cdot\dotsc\cdot A_{X_{k-1}}\cdot B_{X_{k}}$ and observe that for all $\ell,m\in\mathbb{N}_{0}$ with $\tau(m)\leq\ell\leq\tau(m+1)$ the following holds
\begin{align*}
  \label{eqn:observation}
  \frac{|B_{Z_{\ell}}-B_{Z_{\tau(m)}}|}{A_{Z_{\tau(m)}}} &\leq\frac{A_{X_{1}}\cdot\dotsc\cdot A_{X_{\tau(m)}}\cdot\sum_{k=\tau(m)+1}^{\ell}A_{X_{\tau(m)+1}}\cdot\dotsc\cdot A_{X_{k-1}}\cdot|B_{X_{k}}|}{A_{X_{1}}\cdot\dotsc\cdot A_{X_{\tau(m)}}} \\[2pt] &=\sum_{k=\tau(m)+1}^{\ell}\underbrace{A_{X_{\tau(m)+1}}\cdot\dotsc\cdot A_{X_{k-1}}}_{\leq\,1}\cdot\,|B_{X_{k}}|\leq\sum_{k=\tau(m)+1}^{\ell}|B_{X_{k}}|~~\text{a.\,s.} \tag{$\ast\!\ast$}
\end{align*}
Hence, using that $A_{Z_{\tau(0)}}<A_{Z_{\tau(1)}}<\dotsc<A_{Z_{\tau(m)}}$, we obtain
\begin{align*}
  \circled{2} &=\frac{1}{n}\cdot\arcosh\left(1+\frac{1}{2}\cdot\left(\frac{|B_{Z_{n}}|}{A_{Z_{\tau(m)}}}\right)^{2}\right)=\frac{1}{n}\cdot\ln\left(1+\frac{1}{2}\cdot\left(\frac{|B_{Z_{n}}|}{A_{Z_{\tau(m)}}}\right)^{2}+\sqrt{\left(1+\frac{1}{2}\cdot\left(\frac{|B_{Z_{n}}|}{A_{Z_{\tau(m)}}}\right)^{2}\right)^{2}-1}\;\right) \\[3pt] &\leq\frac{1}{n}\cdot\left(\ln(2)+\ln\left(1+\left(\frac{|B_{Z_{n}}|}{A_{Z_{\tau(m)}}}\right)^{2}\right)\right)\leq\frac{1}{n}\cdot\left(\ln(2)+2\cdot\ln\left(1+\frac{|B_{Z_{n}}|}{A_{Z_{\tau(m)}}}\right)\right) \\[10pt] &\leq\frac{1}{n}\cdot\left(\ln(2)+2\cdot\ln\left(1+\frac{|B_{Z_{\tau(1)}}-B_{Z_{\tau(0)}}|}{A_{Z_{\tau(0)}}}+\frac{|B_{Z_{\tau(2)}}-B_{Z_{\tau(1)}}|}{A_{Z_{\tau(1)}}}+\dotsc+\frac{|B_{Z_{\tau(m)}}-B_{Z_{\tau(m-1)}}|}{A_{Z_{\tau(m-1)}}}+\frac{|B_{Z_{n}}-B_{Z_{\tau(m)}}|}{A_{Z_{\tau(m)}}}\right)\right)\,,
\end{align*}
which allows us to use that $\tau(m)\leq n<\tau(m+1)$ and apply (\ref{eqn:observation}), to apply (\ref{eqn:truncating}), and to obtain \renewcommand\qedsymbol{$\square$}
\begin{align*}
\dotsc&\leq\frac{1}{n}\cdot\left(\ln(2)+2\cdot\ln\left(1+\sum_{k=1}^{n}|B_{X_{k}}|\right)\right)\leq\frac{1}{n}\cdot\left(\ln(2)+2\cdot\ln\left(1+\sum_{k=1}^{\tau(m+1)}|B_{X_{k}}|\right)\right) \\[8pt] &\leq\underbrace{\frac{\ln(2)}{n\vphantom{()}}}_{\to\;0}+\,2\cdot\underbrace{\frac{\ln\left(1+\sum_{k=\tau(0)+1}^{\tau(1)}|B_{X_{k}}|\right)+\dotsc+\ln\left(1+\sum_{k=\tau(m)+1}^{\tau(m+1)}|B_{X_{k}}|\right)}{m+1\vphantom{()}}}_{\to\;\mathbb{E}(\ln(1+\sum_{k=1}^{\tau}|B_{X_{k}}|))~\text{a.\,s.~by Lemma \ref{lem:elie}}}\cdot\underbrace{\frac{m+1}{\tau(m)}}_{\substack{\to\;0 \\ \text{a.\,s.}}}\cdot\underbrace{\frac{\tau(m)}{n\vphantom{()}}}_{\substack{\leq\,1 \\ \text{a.\,s.}}}\,\xrightarrow[\text{\raisebox{.5ex}{a.\,s.}}]{~n\to\infty~}\,0\,.\qedhere
\end{align*}
\end{proof}


\subsection{Convergence to the space of ends of the Bass--Serre tree}
\label{sub:cttsoeotbst}

Even though the projections $\pi_{T}(Z_{n})$ do not need to satisfy the Markov property, we are still able to show that they converge a.\,s.~to a random end by applying a result of Cartwright and Soardi, \cite[p.\,820, Theorem]{cs-1989}, which is based on a technique developed by Furstenberg in \cite{furstenberg-1963} and \cite{furstenberg-1971}. The authors consider a random walk $\Phi=(\Phi_{0},\Phi_{1},\dotsc)$ on the automorphism group of a locally finite and infinite tree and prove under a mild assumption on the probability measure that the sequence of vertices obtained by evaluating each automorphism $\Phi_{n}$ at a fixed vertex $v$ converges a.\,s.~to a random end. Their assumption is that the random walk is driven by a \emph{regular Borel probability measure} whose support is \emph{not contained in any amenable subgroup}. However, the proof of \cite[p.\,820, Theorem]{cs-1989} shows that it suffices to assume that the support is not contained in any amenable \emph{closed} subgroup. Given that $1<p<q$, this result can be immediately applied to our setting.

\begin{lem}
\label{lem:anytree}
Let $1<p<q$. Then, the projections $\pi_{T}(Z_{n})$ converge a.\,s.~to a random end $\xi\in\partial T$.
\end{lem}

\vspace{-10pt}\begin{proof}
Since the group $G$ acts on the tree $T$, we may consider the group homomorphism $\varphi:G\to\Aut(T)$ associated to this action. The automorphism group $\Aut(T)$ is endowed with the topology of pointwise convergence. Since $G$ is discrete, $\varphi$ is certainly measurable. The pointwise images $(\varphi(Z_{0}),\varphi(Z_{1}),\dotsc)$ constitute a random walk on $\Aut(T)$ that satisfies the assumption of \cite[p.\,820, Theorem]{cs-1989}. Indeed, the random walk on $\Aut(T)$ is driven by the pushforward \emph{Borel probability measure} $\varphi_{\ast}(\mu)$. Because $\Aut(T)$ is a locally compact Hausdorff space with a countable base, see e.\,g.~\cite[\S 2]{woess-1991}, every Borel probability measure on $\Aut(T)$ is \emph{regular}, see e.\,g.~\cite[Proposition 7.2.3]{cohn-2013}, and so is $\varphi_{\ast}(\mu)$. It remains to show that the support of the latter is not contained in any amenable closed subgroup.

Observe that the support of $\varphi_{\ast}(\mu)$ generates the subgroup $\varphi(G)\leq\Aut(T)$. Since $\varphi(G)$ acts transitively on~$T$ and does not fix an end, every closed subgroup that contains $\varphi(G)$ has these two properties as well and is therefore not amenable, see \cite[Theorem 2]{nebbia-1988}. In other words, the support of $\varphi_{\ast}(\mu)$ is \emph{not contained in any amenable closed subgroup} of $\Aut(T)$. Now, \cite[p.\,820, Theorem]{cs-1989} yields that the sequence obtained by evaluating each automorphism $\varphi(Z_{n})$ at a fixed vertex $v$ converges a.\,s.~to a random end $\xi\in\partial T$. Since $\pi_{T}(Z_n)=Z_{n}B=\varphi(Z_{n})(B)$, setting $v:=B$ completes the proof.
\end{proof}


\subsection{\texorpdfstring{Construction of $\mu$-boundaries}{Construction of mu-boundaries}}
\label{sub:comb}

Resuming Sections \ref{sub:cttbothp} and \ref{sub:cttsoeotbst}, we may formulate the following theorem.

\begin{thm}[``convergence theorem'']
Let $Z=(Z_{0},Z_{1},\dotsc)$ be a random walk on a non-amenable Baumslag--Solitar group $G=\BS(p,q)$ with $1<p<q$ and increments $X_{1},X_{2},\dotsc$ of finite first moment. Then, the projections $\pi_{T}(Z_{n})$ converge a.\,s.~to a random end $\xi\in\partial T$. Moreover, depending on the vertical drift $\delta$, we distinguish three cases:
\begin{enumerate}
\item If $\delta>0$, then the projections $\pi_{\mathbb{H}}(Z_{n})$ converge a.\,s.~to $\infty\in\partial\mathbb{H}$.
\item If $\delta<0$, then the projections $\pi_{\mathbb{H}}(Z_{n})$ converge a.\,s.~to a random element $r\in\mathbb{R}=\partial\mathbb{H}\smallsetminus\sset{\infty}$. 
\item If $\delta=0$ and $\ln(A_{X_{1}})$ has finite second moment and there is an $\varepsilon>0$ such that $\ln(1+|B_{X_{1}}|)$ has finite $(2+\varepsilon)$-th moment, then projections $\pi_{\mathbb{H}}(Z_{n})$ have sublinear speed.
\end{enumerate} 
\end{thm}\pagebreak

So, let us assume that $1<p<q$ and that the increments $X_{1},X_{2},\dotsc$ have finite first moment. We may therefore consider the map $\bnd_{\partial T}:\Omega\to\partial T$ and, in the special case that $\delta<0$, also the map $\bnd_{\mathbb{R}}:\Omega\to\mathbb{R}$, defined almost everywhere, assigning to a trajectory $\omega=(x_{0},x_{1},\dotsc)\in\Omega$ the limit
\begin{equation*}
  \bnd_{\partial T}(\omega):=\lim_{n\to\infty}\pi_{T}(x_{n})\in\partial T \quad\text{and}\quad\bnd_{\mathbb{R}}(\omega):=\lim_{n\to\infty}\pi_{\mathbb{H}}(x_{n})\in\mathbb{R}\,.
\end{equation*}
The topological spaces $\partial T$ and $\mathbb{R}$ are endowed with their Borel $\sigma$-algebras $\mathcal{B}_{\partial T}$ and $\mathcal{B}_{\mathbb{R}}$. Even though the maps $\bnd_{\partial T}$ and $\bnd_{\mathbb{R}}$ are only defined almost everywhere, they are measurable in the sense that the preimages of measurable sets are measurable. Given $\bnd_{\partial T}$ and $\bnd_{\mathbb{R}}$, we may construct their product map $\bnd_{\partial T \times\mathbb{R}}:\Omega\to\partial T \times\mathbb{R}$. It is measurable with respect to the product $\sigma$-algebra $\mathcal{B}_{\partial T}\otimes\mathcal{B}_{\mathbb{R}}$. Because both $\partial T$ and $\mathbb{R}$ are metrisable and separable topological spaces, it is not hard to see that the product $\sigma$-algebra $\mathcal{B}_{\partial T}\otimes\mathcal{B}_{\mathbb{R}}$ agrees with the Borel $\sigma$-algebra $\mathcal{B}_{\partial T \times\mathbb{R}}$, see e.\,g.~\cite[Appendix M.10]{billingsley-1999}. The pushforward probability measures $\nu_{\partial T}:=(\bnd_{\partial T})_{\ast}(\mathbb{P})$ and $\nu_{\partial T\times\mathbb{R}}:=(\bnd_{\partial T\times\mathbb{R}})_{\ast}(\mathbb{P})$ on the respective measurable spaces are called the \emph{hitting measures}. Since $\partial T$ and $\mathbb{R}$, and therefore also $\partial{T}\times\mathbb{R}$, are Polish spaces, Example \ref{exa:polish} implies that $(\partial T ,\mathcal{B}_{\partial T},\nu_{\partial T})$ and $(\partial T\times\mathbb{R},\mathcal{B}_{\partial T\times\mathbb{R}},\nu_{\partial T\times\mathbb{R}})$ are Lebesgue--Rohlin spaces. The maps $\bnd_{\partial{T}}$ and $\bnd_{\partial T\times\mathbb{R}}$ are homomorphisms and, by construction, they are \circled{1} $\sim$-invariant.

Each of the topological spaces $\partial T$ and $\mathbb{R}$ is endowed with a continuous $G$-action. The one on $\partial T$ is induced by the left-multiplication $g(hB):=(gh)B$ on~$T$. More precisely, recall that ends are infinite reduced paths that start in $B$. The pointwise left-multiplication maps every such path $\xi\in\partial T$ to some other path that need not start in $B$ anymore. The end $g\xi\in\partial T$ is obtained by connecting~$B$ to the initial vertex of this path and reducing the concatenation if necessary. The $G$-action on $\mathbb{R}$ is induced by the isometric $G$-action on~$\mathbb{H}$ that we addressed in Section~\ref{sub:ptthp}. In light of the representation of the elements of $\isom(\mathbb{H})$ as rational functions, we can also evaluate them on the boundary $\partial\mathbb{H}$ and finally observe that the isometries associated to the elements of~$G$ leave the subset $\mathbb{R}\subseteq\partial\mathbb{H}$ invariant. The $G$-actions on $\partial T$ and~$\mathbb{R}$ induce a componentwise $G$-action on the product $\partial T\times\mathbb{R}$, which is also continuous.

All three $G$-actions are measurable with respect to the Borel $\sigma$-algebras and, since they map null sets $A$ to null sets $gA$, they remain measurable when we proceed to the completions. In particular, the spaces $(\partial T ,\mathcal{B}_{\partial T},\nu_{\partial T})$ and $(\partial T\times\mathbb{R},\mathcal{B}_{\partial T\times\mathbb{R}},\nu_{\partial T\times\mathbb{R}})$ are endowed with measurable $G$-actions and, by construction, the maps $\bnd_{\partial{T}}$ and $\bnd_{\partial T\times\mathbb{R}}$ are \circled{2} $G$-equivariant. We have thus derived the following lemma.

\begin{lem}
\label{lem:muboundary}
For any vertical drift $\delta$, in particular for $\delta\geq 0$, the Lebesgue--Rohlin space $(\partial T ,\mathcal{B}_{\partial T},\nu_{\partial T})$ endowed with the homomorphism $\bnd_{\partial T}:\Omega\to\partial T$ is a $\mu$-boundary. If $\delta<0$, then $(\partial T\times\mathbb{R},\mathcal{B}_{\partial T\times\mathbb{R}},\nu_{\partial T\times\mathbb{R}})$ endowed with $\bnd_{\partial T\times\mathbb{R}}:\Omega\to\partial T\times\mathbb{R}$ is also a $\mu$-boundary.
\end{lem}

Before we will use Kaimanovich's strip criterion to show that the above $\mu$-boundaries are maximal, we analyse the hitting measures. This requires a preliminary observation.

\begin{lem}
\label{lem:minimal}
The $G$-actions on $\partial T$ and $\mathbb{R}$, as well as the componentwise $G$-action on the product $\partial T\times\mathbb{R}$, are topologically minimal, i.\,e.~each orbit is dense. Because all three spaces are infinite and Hausdorff, this implies that each orbit is infinite.
\end{lem}

\vspace{-10pt}\begin{proof}
Consider the $G$-action on $\partial T$. Choose an end $\xi\in\partial T$ and a non-empty open subset $S\subseteq\partial T$. We shall construct an element $g\in G$ such that $g\xi\in S$. Because $S$ is non-empty and open, all ends with a certain finite initial piece belong to $S$. In other words, there is an element $h\in G$ such that all ends that start in $B$ and traverse the vertex $hB$ are contained in $S$. If we set $g:=h\in G$, then the end $g\xi\in\partial T$ will have the correct finite initial piece unless cancellation takes place. In the latter case, we set $g:=hb\in G$ instead. Since $|p|\neq 1$ and $|q|\neq 1$, cancellation will take place in at most one of the two cases, which proves the first assertion.

Next, consider the $G$-action on $\mathbb{R}$, an element $r\in\mathbb{R}$, and a non-empty open subset $S\subseteq\mathbb{R}$. Because $S$ is non-empty and open, there are $s\in\mathbb{R}$ and $\varepsilon>0$ such that the interval $(s-\varepsilon,s+\varepsilon)$ is contained in~$S$. We assume that $1<p<q$, so we can find integers $k_{1},k_{2}\in\mathbb{Z}$ with $k_{1}<0$ such that $q^{k_{1}}p^{-k_{1}}<\varepsilon$ and the elements $q^{k_{1}}p^{-k_{1}}(r+k_{2})$ and $q^{k_{1}}p^{-k_{1}}(r+k_{2}+1)$ are both contained in $(s-\varepsilon,s+\varepsilon)$. Therefore, we set $g:=a^{k_{1}}b^{k_{2}}\in G$ to obtain
\begin{equation*}
  gr=a^{k_{1}}b^{k_{2}}r=a^{k_{1}}(r+k_{2})=q^{k_{1}}p^{-k_{1}}(r+k_{2})\in(s-\varepsilon,s+\varepsilon)\subseteq S\,.
\end{equation*}
Finally, consider the $G$-action on the product $\partial T\times\mathbb{R}$, an element $(\xi,r)\in\partial T\times\mathbb{R}$, and a non-empty open subset $S\subseteq\partial T\times\mathbb{R}$. Because $S$ is non-empty and open, there are non-empty and open subsets $S_{1}\subseteq\partial T$ and $S_{2}\subseteq\mathbb{R}$ such that $S_{1}\times S_{2}$ is contained in $S$. We shall now construct an element $g\in G$ such that both $g\xi\in S_{1}$ and $gr\in S_{2}$. Look at the tree component first. We already know that there is an element $h\in G$ such that all ends that start in $B$ and traverse the vertex $hB$ are contained in $S_{1}$. Let $k_{0}\in\sset{0,1}$, whichever ensures that the reduced path from $B$ to the vertex $hb^{k_{0}}a^{-1}B$ traverses the vertex $hB$. Now, look at the real component. We can find integers $k_{1},k_{2}\in\mathbb{Z}$ with $k_{1}<0$ such that the elements $a^{k_{1}}b^{k_{2}}r$ and $a^{k_{1}}b^{k_{2}+1}r$ are both contained in $(hb^{k_{0}})^{-1}S_{2}$. Back to the tree component, we choose $k_{3}\in\sset{0,1}$ such that the end $hb^{k_{0}}a^{k_{1}}b^{k_{2}+k_{3}}\xi$ traverses the vertex $hb^{k_{0}}a^{k_{1}}B$. Then, by construction, it also traverses the vertex $hB$. We set $g:=hb^{k_{0}}a^{k_{1}}b^{k_{2}+k_{3}}\in G$ to obtain $g\xi\in S_{1}$ and $gr\in S_{2}$.
\end{proof}

Given Lemma~\ref{lem:minimal} and the $\mu$-stationarity and quasi-invariance of the hitting measures, it is well-known that the latter are non-atomic and have full support. Indeed, if there were atoms, then we could choose an atom~$\xi$ of maximal mass. Because the respective hitting measure $\nu$ is $\mu$-stationary, the value $\nu(\xi)$ is a convex combination of all values $\nu(g^{-1}\xi)$ with $g\in\supp(\mu)$. Therefore, each $\nu(g^{-1}\xi)$ must be equal to~$\nu(\xi)$. Iteration of this procedure yields that the equality does not only hold for every $g\in\supp(\mu)$ but also for every $g$ in the semigroup generated by $\supp(\mu)$, i.\,e.~for every $g\in G$. Since the orbit $G\xi$ is infinite, this contradicts the finiteness of the hitting measure $\nu$. Concerning the assertion of full support, if there was a non-empty open null set~$S$, then the topological minimality of the $G$-action would imply that the translates $gS$ with $g\in G$ form a countable covering of the whole space with null sets, which is a contradiction. For further details, see e.\,g.~\cite[Lemma 3.4]{woess-1989} and \cite[Lemma 2.2 and 2.3]{mns-2017}.


\subsection{Proof of the main result}

We are now ready to prove our main result, Theorem \ref{thm:identification} announced in the \hyperlink{sec:introduction}{introduction}. It identifies the Poisson--Furstenberg boundary of random walks $Z=(Z_{0},Z_{1},\dotsc)$ on non-amenable Baumslag--Solitar groups $G=\BS(p,q)$ with $1<p<q$ and increments $X_{1},X_{2},\dotsc$ of finite first moment.

\begin{proof}[Proof of Theorem \ref{thm:identification}]
We seek to apply the strip criterion, Theorem \ref{thm:strip}. By Lemma \ref{lem:entropy}, the probability measure $\mu$ driving the random walk has finite entropy. By Lemma~\ref{lem:muboundary}, the Lebesgue--Rohlin space $(\partial T ,\mathcal{B}_{\partial T},\nu_{\partial T})$ is a $\mu$-boundary. If $Z$ has negative vertical drift $\delta<0$, then $(\partial T\times\mathbb{R},\mathcal{B}_{\partial T\times\mathbb{R}},\nu_{\partial T\times\mathbb{R}})$ is also a $\mu$-boundary. Let us consider the case $\delta<0$ first. We thus take the $\mu$-boundary $(\partial T\times\mathbb{R},\mathcal{B}_{\partial T\times\mathbb{R}},\nu_{\partial T\times\mathbb{R}})$ and the {$\check{\mu}$-boundary} $(\partial T ,\mathcal{B}_{\partial T},\hspace*{0.9pt}\check{\hspace*{-0.9pt}\nu}_{\partial T})$. Here, $\hspace*{0.9pt}\check{\hspace*{-0.9pt}\nu}_{\partial T}$ denotes the hitting measure of the pointwise projection of the random walk $\check{Z}=(\check{Z}_{0},\check{Z}_{1},\dotsc)$ driven by the reflected probability measure $\check{\mu}$ to the tree $T$.

Next, we define gauges and strips. Let $S:=\sset{a,b}\subseteq G$ be the standard generating set and define gauges $\mathcal{G}_{k}:=\set{g\in G}{\dwordS(1,g)\leq k}$, i.\,e.~the gauges exhaust the group $G$ with balls centred at the identity element $1\in G$ and the gauge function $|\cdot|=|\cdot|_{\mathcal{G}}$ is essentially the distance to $1$ with respect to the word metric $\dwordS$.

Since the hitting measures are non-atomic, see Section~\ref{sub:comb}, we know that $\hspace*{0.9pt}\check{\hspace*{-0.9pt}\nu}_{\partial T}\otimes\nu_{\partial T\times\mathbb{R}}$-almost every pair of points $(\xi_{-},(\xi_{+},r_{+}))\in\partial T \times(\partial T \times\mathbb{R})$ has distinct ends $\xi_{-},\xi_{+}\in\partial T$. In this situation, we may connect $\xi_{-}$ and~$\xi_{+}$ by a doubly infinite reduced path $v:\mathbb{Z}\to T$ and define the strip $\mathcal{S}(\xi_{-},(\xi_{+},r_{+}))$ as follows. It consists of all $g\in G$ in the full $\pi_{T}$-preimage of the path, i.\,e.~the image $\pi_{T}(g)$ is contained in~$v(\mathbb{Z})$, with the additional property that the real part $\Re(\pi_{\mathbb{H}}(g))$ has minimal distance to $r_{+}\in\mathbb{R}$ among all real parts $\Re(\pi_{\mathbb{H}}(h))$ with $h\in gB$, see the left-hand side of Figure \ref{fig:strips}. To all remaining pairs we assign the whole of $G$ as a strip. This way, the map $\mathcal{S}$ becomes measurable and $G$-equivariant. Since the hitting measures have full support, see Section~\ref{sub:comb} again, it is not hard to see that a random strip contains the identity element $1\in G$ with positive probability, i.\,e.~the map $\mathcal{S}$ satisfies the inequality of Remark~\ref{rem:removeg}. So, it suffices to verify the following convergence for an arbitrary pair $(\xi_{-},(\xi_{+},r_{+}))\in\partial T \times(\partial T\times\mathbb{R})$ with distinct ends $\xi_{-},\xi_{+}\in\partial T$,
\begin{figure}
\begin{center}
  \includegraphics{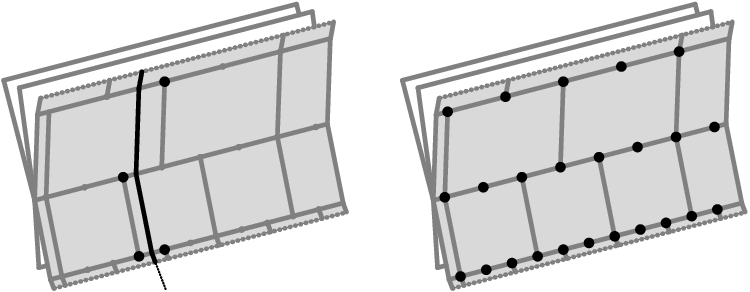}
  \begin{picture}(0,0)
    \put(-286,-7){$r_{+}\in\mathbb{R}$}
  \end{picture}
\end{center}
\caption{Strips for the cases $\delta\neq 0$ (left) and $\delta=0$ (right).}
\label{fig:strips}
\end{figure}
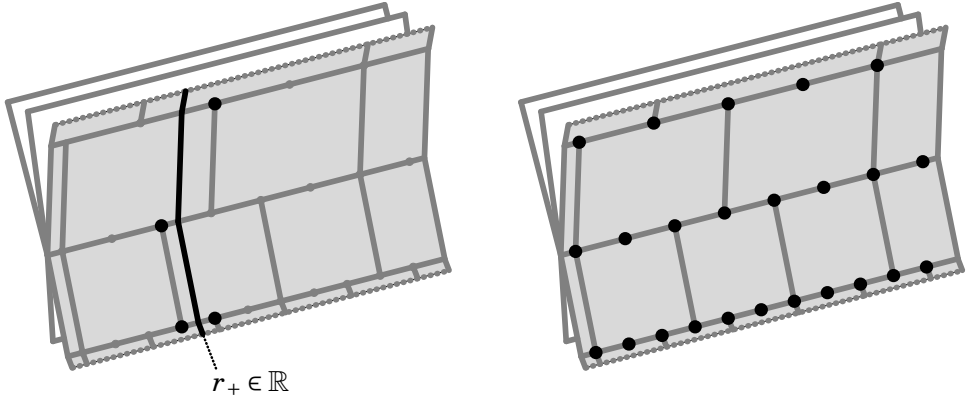
\begin{equation*}
  \frac{1}{n}\cdot\ln\left(\card\left(\mathcal{S}(\xi_{-},(\xi_{+},r_{+}))\cap\mathcal{G}_{|Z_{n}|}\right)\right)\,\xrightarrow[\text{\raisebox{.5ex}{a.\,s.}}]{~n\to\infty~}\,0\,.
\end{equation*}
The strip $\mathcal{S}(\xi_{-},(\xi_{+},r_{+}))$ intersects the gauge $\mathcal{G}_{|Z_{n}|}$ in at most $2\cdot|Z_{n}|+1$ many cosets from $G/B$ and each of them contains at most two elements of the strip. Therefore,
\begin{equation*}
  \frac{1}{n}\cdot\ln\left(\card\left(\mathcal{S}(\xi_{-},(\xi_{+},r_{+}))\cap\mathcal{G}_{|Z_{n}|}\right)\right)\leq\frac{\ln\big((2\cdot|Z_{n}|+1)\cdot 2\big)}{n}=\frac{\ln\big((2\cdot\dwordS(1,Z_{n})+1)\cdot 2\big)}{n}\,\xrightarrow[\text{\raisebox{.5ex}{a.\,s.}}]{~n\to\infty~}\,0\,.
\end{equation*}
In the final step above, we used that $X_{1}$ has finite first moment. Indeed,
\begin{equation*}
  \frac{1}{n}\cdot\dwordS(1,Z_{n})=\frac{1}{n}\cdot\dwordS(1,X_{1}\cdot\dotsc\cdot X_{n})\leq\frac{1}{n}\cdot\sum_{k=1}^{n}\dwordS(1,X_{k})\,\xrightarrow[\text{\raisebox{.5ex}{a.\,s.}}]{~n\to\infty~}\,\mathbb{E}(\dwordS(1,X_{1}))\,,
\end{equation*}
from where we may first conclude that the sequence $\frac{1}{n}\cdot\dwordS(1,Z_{n})$ is a.\,s.~bounded and second that the sequence $\frac{1}{n}\cdot\ln((2\cdot\dwordS(1,Z_{n})+1)\cdot 2)$ converges a.\,s.~to 0. So, we can finally apply the strip criterion and obtain that $(\partial T\times\mathbb{R},\mathcal{B}_{\partial T\times\mathbb{R}},\nu_{\partial T\times\mathbb{R}})$ is isomorphic to the Poisson--Furstenberg boundary. Vice versa, if $Z$ has positive vertical drift $\delta>0$, then the same argument yields that $(\partial T ,\mathcal{B}_{\partial T},\nu_{\partial T})$ is isomorphic to the Poisson--Furstenberg boundary.

It remains to consider the driftless case, i.\,e.~$\delta=0$. Then, both $\mu$ and $\check{\mu}$ are driftless and there is no natural candidate for a real number that determines the horizontal position of the strip. But the fact that the projections $\pi_{\mathbb{H}}(Z_{n})$ have sublinear speed, see Lemma~\ref{lem:zerodrifthyp}, allows us to solve this issue. More precisely, take the $\mu$-boundary $(\partial T ,\mathcal{B}_{\partial T},\nu_{\partial T})$ and the $\check{\mu}$-boundary $(\partial T ,\mathcal{B}_{\partial T},\hspace*{0.9pt}\check{\hspace*{-0.9pt}\nu}_{\partial T})$ and define gauges
\begin{equation*}
  \mathcal{G}_{k}:=\left\{\,g\in G\,:\,\dtree(\pi_{T}(1),\pi_{T}(g))\leq k^{2}~\text{and}~\dhyp(\pi_{\mathbb{H}}(1),\pi_{\mathbb{H}}(g))\leq k\right\}\,.
\end{equation*}
Again, $\hspace*{0.9pt}\check{\hspace*{-0.9pt}\nu}_{\partial T}\otimes\nu_{\partial T}$-almost every pair of points $(\xi_{-},\xi_{+})\in\partial T \times\partial T$ has distinct ends $\xi_{-},\xi_{+}\in\partial T$, which we may connect by a doubly infinite reduced path $v:\mathbb{Z}\to T$. Let $\mathcal{S}(\xi_{-},\xi_{+})$ be the full $\pi_{T}$-preimage of the path, i.\,e.~the set consisting of all $g\in G$ such that $\pi_{T}(g)$ is contained in $v(\mathbb{Z})$, see the right-hand side of Figure \ref{fig:strips}. Again, to all remaining pairs we assign the whole of $G$ as a strip. This way, the map $\mathcal{S}$ becomes measurable, $G$-equivariant, and satisfies the inequality of Remark \ref{rem:removeg}. Now, pick an arbitrary pair $(\xi_{-},\xi_{+})\in\partial T \times\partial T$ with distinct ends $\xi_{-},\xi_{+}\in\partial T$. We claim that
\begin{equation*}
  \frac{1}{n}\cdot\ln\left(\card\left(\mathcal{S}(\xi_{-},\xi_{+})\cap\mathcal{G}_{|Z_{n}|}\right)\right)\leq\frac{\ln\big((2\cdot|Z_{n}|^{2}+1)\cdot\exp(|Z_{n}|+2)\big)}{n}=\underbrace{\frac{\ln\left(2\cdot|Z_{n}|^{2}+1\right)}{n}}_{\circled{1}}+\underbrace{\frac{|Z_{n}|+2}{n}}_{\circled{2}}\,.
\end{equation*}
Indeed, the inequality holds for a similar reason as above. The strip $\mathcal{S}(\xi_{-},\xi_{+})$ intersects the gauge $\mathcal{G}_{|Z_{n}|}$ in at most $2\cdot|Z_{n}|^{2}+1$ many cosets from $G/B$. Slightly more involved is the observation that each of them contains at most $\exp(|Z_{n}|+2)$ many elements of the gauge. Fix a coset $gB$. The projections $\pi_{\mathbb{H}}(h)$ of the elements $h\in gB$ are located on a horizontal line $L\subseteq\mathbb{H}$ with imaginary part $y:=\Im(\pi_{\mathbb{H}}(g))$. One necessary condition for such an element $h\in gB$ to be contained in the gauge $\mathcal{G}_{|Z_{n}|}$ is that the projection $\pi_{\mathbb{H}}(h)$ is contained in the closed disc $D:=\set{z\in\mathbb{H}}{\dhyp(i,z)\leq|Z_{n}|}\subseteq\mathbb{H}$. If $L\cap D$ is empty, then the coset~$gB$ does not contain any element of the gauge and we are done. Otherwise, there is a unique $x\in\mathbb{R}$ with $x\geq 0$ such that $L\cap D$ is the horizontal line between $z_{1}:=-x+iy$ and $z_{2}:=x+iy$, see Figure \ref{fig:shortline}. The projections $\pi_{\mathbb{H}}(h)$ with $h\in gB$ have the property that the real parts $\Re(\pi_{\mathbb{H}}(h))$ and $\Re(\pi_{\mathbb{H}}(hb))$ differ exactly by $y$. So, $L\cap D$ contains at most $1+\frac{2x}{y}$ many of them. Let us now estimate $1+\frac{2x}{y}$ in terms of~$|Z_{n}|$. Since $z_{1}$ and $z_{2}$ are both contained in $D$, their distance is at most $2\cdot|Z_{n}|$. Therefore,
\begin{equation*}
  2\cdot|Z_{n}|\geq\dhyp(z_{1},z_{2})=\arcosh\left(1+\frac{|z_{2}-z_{1}|^{2}}{2\cdot\Im(z_{1})\Im(z_{2})}\right)=\arcosh\left(1+\frac{2x^{2}}{y^{2}}\right)\geq\ln\left(1+\frac{2x^{2}}{y^{2}}\right)\,.
\end{equation*}
And, in particular,
\begin{gather*}
\exp(2\cdot|Z_{n}|)>\frac{2x^{2}}{y^{2}}\,,\iff\exp\left(|Z_{n}|+\frac{1}{2}\cdot\ln(2)\right)>\frac{2x}{y}\,,\implies\exp(|Z_{n}|+2)>1+\frac{2x}{y}\,.
\end{gather*}
So, the coset $gB$ contains strictly fewer than $\exp(|Z_{n}|+2)$ elements of the gauge. We will now show that both summands \circled{1} and \circled{2} converge a.\,s.~to $0$, which will complete the proof. Let us first observe that
\begin{equation*}
  \label{eqn:gaugefunctionestimate} \tag{$\ast$}
  |Z_{n}|-1\leq\max\left\{\,\dhyp(\pi_{\mathbb{H}}(1),\pi_{\mathbb{H}}(Z_{n})),\sqrt{\dtree(\pi_{T}(1),\pi_{T}(Z_{n}))}\,\right\}\leq\max\left\{\,\dhyp(\pi_{\mathbb{H}}(1),\pi_{\mathbb{H}}(Z_{n})),\sqrt{\dwordS(1,Z_{n})}\,\right\}\,.
\end{equation*}
Concerning \circled{1}, we deduce from (\ref{eqn:gaugefunctionestimate}) and Lemma \ref{lem:lookingfromfront} that $|Z_{n}|$ is at most $\max\sset{\ell_{a},\ell_{b},1}\cdot\dwordS(1,Z_{n})+1$ and finally obtain by the same argument as above
\begin{equation*}
  \circled{1}=\frac{\ln\left(2\cdot|Z_{n}|^{2}+1\right)}{n}\leq\frac{\ln\left(2\cdot(\max\sset{\ell_{a},\ell_{b},1}\cdot\dwordS(1,Z_{n})+1)^{2}+1\right)}{n}\,\xrightarrow[\text{\raisebox{.5ex}{a.\,s.}}]{~n\to\infty~}\,0\,.\
\end{equation*}
On the other hand, concerning \circled{2}, we apply (\ref{eqn:gaugefunctionestimate}) and Lemma \ref{lem:zerodrifthyp} to obtain \renewcommand\qedsymbol{$\square$}
\begin{equation*}
  \circled{2}=\frac{|Z_{n}|+2}{n}\leq\frac{\max\left\{\,\dhyp(\pi_{\mathbb{H}}(1),\pi_{\mathbb{H}}(Z_{n})),\sqrt{\dwordS(1,Z_{n})}\,\right\}+3}{n}\,\xrightarrow[\text{\raisebox{.5ex}{a.\,s.}}]{~n\to\infty~}\,0\,.\qedhere
\end{equation*}
\end{proof}\pagebreak

\begin{figure}
\begin{center}
  \includegraphics{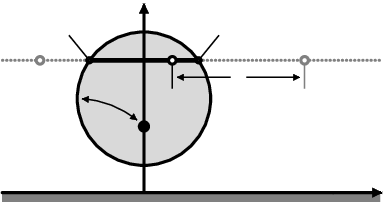}
  \begin{picture}(0,0)
    \put(-114,33){$i$}
    \put(-77,58){$y$}
    \put(-2,65){$L$}
    \put(-89,25){$D$}
    \put(-195,86){$z_{1}=-x+iy$}
    \put(-97,86){$z_{2}=x+iy$}
  \end{picture}
\end{center}
\caption{The horizontal line $L$, the closed disc $D$, and their intersection $L\cap D$.}
\label{fig:shortline}
\end{figure}
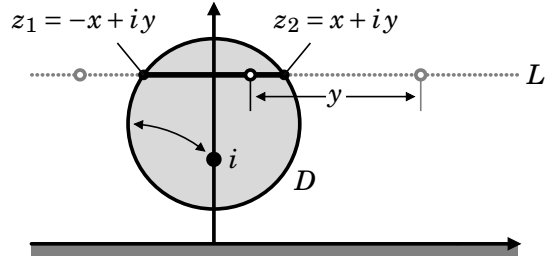
 
%
\appendix
\renewcommand{\leftmark}{APPENDIX: THE REMAINING NON-AMENABLE CASES}
\setcounter{section}{3}
\setcounter{rem}{0}
\renewcommand{\thesection}{A}
\section*{Appendix: The remaining non-amenable cases}
\hypertarget{sec:appendix}{}
Recall from Section \ref{sub:aobsg} that a Baumslag--Solitar group $\BS(p,q)$ is non-amenable if and only if neither $|p|=1$ nor $|q|=1$. Until now, we have only identified the Poisson--Furstenberg boundary for random walks on non-amenable Baumslag--Solitar groups $\BS(p,q)$ with $1<p<q$. Replacing one of the two generators by its inverse yields $\BS(p,q)\cong\BS(q,p)\,$ and $\BS(p,q)\cong\BS(-p,-q)$. So, in order to cover the remaining non-amenable cases, it suffices to consider $1<p<-q$ and $1<p=|q|$. Below, we explain how to adjust our methods to obtain similar results for these cases.


\subsection{Action by suitable isometries on the hyperbolic plane}
\label{sub:absiothp}

Assume that $G=\BS(p,q)$ with $1<p<-q$. The definition of the tree $T$ and the level function $\lambda$ remain the same and even Remark~\ref{rem:neighbours} can be adapted, replacing $q$ by $|q|$. Recall that, in Section~\ref{sub:ptthp}, we first constructed the group homomorphism $\pi_{\isom(\mathbb{H})}:G\to\isom(\mathbb{H})$ and then used it to define the projection $\pi_{\mathbb{H}}:G\to\mathbb{H}$. This is precisely what we are going to do again; but, this time, with another isometry $\pi_{\isom(\mathbb{H})}(a)$. Let $\pi_{\isom(\mathbb{H})}:\sset{a,b}\to\isom(\mathbb{H})$ be the map given by $\pi_{\isom(\mathbb{H})}(a):=\big(z\mapsto\frac{|q|}{p}\cdot(-\widebar{z})\big)$ and $\pi_{\isom(\mathbb{H})}(b):=(z\mapsto z+1)$. It follows from von Dyck's theorem that this map can be uniquely extended to a group homomorphism $\pi_{\isom(\mathbb{H})}:G\to\isom(\mathbb{H})$. Now, we define $\pi_{\mathbb{H}}:G\to\mathbb{H}$ by $\pi_{\mathbb{H}}(g):=\pi_{\isom(\mathbb{H})}(g)(i)$. 

\begin{lem}
For every $g\in G$ the point $\pi_{\mathbb{H}}(ga)\in\mathbb{H}$ is above the point $\pi_{\mathbb{H}}(g)\in\mathbb{H}$; the two points have the same real part and their distance in the hyperbolic plane is $\ell_{a}:=\ln\big(\frac{|q|}{p}\big)$. The point $\pi_{\mathbb{H}}(gb)\in\mathbb{H}$ is either to the right or to the left of the point $\pi_{\mathbb{H}}(g)\in\mathbb{H}$ depending on whether the level $\lambda(g)$ is even or odd; in any case, the two points have the same imaginary part and their distance in the hyperbolic plane is $\ell_{b}:=\ln\big(\frac{3+\sqrt{5}\,}{2}\big)$.
\end{lem}

\vspace{-10pt}\begin{proof}[Proof sketch]
The proof is similar to the one of Lemma \ref{lem:lookingfromfront}, and we shall only address the differences. The points $\pi_{\mathbb{H}}(gb)\in\mathbb{H}$ and $\pi_{\mathbb{H}}(g)\in\mathbb{H}$ are obtained by evaluating $\pi_{\isom(\mathbb{H})}(g)$ at $\pi_{\mathbb{H}}(b)\in\mathbb{H}$ and $\pi_{\mathbb{H}}(1)\in\mathbb{H}$. Again, the image $\pi_{\isom(\mathbb{H})}(g)$ is a composition of $\pi_{\isom(\mathbb{H})}(a^{\pm1})$ and $\pi_{\isom(\mathbb{H})}(b^{\pm1})$. While $\pi_{\isom(\mathbb{H})}(b^{\pm1})$ are translations, each occurrence of $\pi_{\isom(\mathbb{H})}(a^{\pm1})$ yields both a dilation and a reflection at the imaginary axis. This implies that the point $\pi_{\mathbb{H}}(gb)\in\mathbb{H}$ is to the right of the point $\pi_{\mathbb{H}}(g)$ if and only if the number of occurrences of $\pi_{\isom(\mathbb{H})}(a^{\pm1})$ is even, which is the case if and only if $\lambda(g)$ is even. 
\end{proof}

Using this projection, and replacing $q$ by $|q|$ wherever it is necessary, we may repeat most of the arguments from Section \ref{sec:identification}. For example, the definitions of the imaginary part $A_{g}$ and the real part~$B_{g}$ now yield the equation $\ln(A_{g})=\ln\big(\frac{|q|}{p}\big)\cdot\lambda(g)$. In order to identify the Poisson--Furstenberg boundary geometrically, we first showed that the pointwise projections of the random walk to $\mathbb{H}$ and $T$ converge a.\,s.~to random elements in the respective boundaries.

While the proof of Lemma \ref{lem:positivedrifthyp} for $\delta>0$ can be adapted, the one of Lemma \ref{lem:negativedrifthyp} for $\delta<0$ requires some additional work. We have to show that the real parts $B_{Z_{n}}$ converge a.\,s.~to a random element $r\in\mathbb{R}$. In the original proof, we used that $A_{Z_{n}}=A_{X_{1}}\cdot\dotsc\cdot A_{X_{n}}$ and $B_{Z_{n}}=\sum_{k=1}^{n}C_{k}$ with $C_{k}:=A_{X_{1}}\cdot\dotsc\cdot A_{X_{k-1}}\cdot B_{X_{k}}$. The first formula for $A_{Z_{n}}$ remains true. However, the second one for $B_{Z_{n}}$ does not because we are now in a situation where not only the scaling but also the direction of the next horizontal increment depends on the current level. Instead, we obtain that $B_{Z_{n}}=\sum_{k=1}^{n}C_{k}$ with $C_{k}:=\varepsilon_{X_{1}}\cdot A_{X_{1}}\cdot\dotsc\cdot\varepsilon_{X_{k-1}}\cdot A_{X_{k-1}}\cdot B_{X_{k}}$ where $\varepsilon_{g}:=1$ if $\lambda(g)$ is even and $\varepsilon_{g}:=-1$ if $\lambda(g)$ is odd. This allows us to apply Cauchy's root test as in the proof of Lemma \ref{lem:negativedrifthyp}. Moreover, the proofs of Lemmas \ref{lem:zerodrifthyp} and \ref{lem:elie} for $\delta=0$ can also be adapted because the estimates are not in terms of the actual horizontal increments but only of their absolute values.

Concerning the boundary $\partial T$, it suffices to observe that the proof of Lemma \ref{lem:anytree} only requires the property that the subgroup $\varphi(G)\leq\Aut(T)$ acts transitively on $T$ and does not fix an end, which is always the case unless $|p|=1$ or $|q|=1$. Therefore, it still shows that the projections $\pi_{T}(Z_{n})$ converge a.\,s.~to a random end in~$\partial T$. As in Lemma \ref{lem:minimal}, we can show that the $G$-actions on $\partial T$ and $\partial T\times\mathbb{R}$ are topologically minimal, whence the hitting measures~$\nu_{\partial T}$ and $\nu_{\partial T\times\mathbb{R}}$ are non-atomic and have full support. This allows us to adapt the proof of Theorem \ref{thm:identification} and to obtain the following version of the~identification theorem.

\begin{thm}[``identification theorem'' for $\mathbf{1<p<-q}$]
Let $Z=(Z_{0},Z_{1},\dotsc)$ be a random walk on a non-amenable Baumslag--Solitar group $G=\BS(p,q)$ with $1<p<-q$ and increments $X_{1},X_{2},\dotsc$ of finite first moment. Depending on the vertical drift $\delta$, we distinguish three cases: 
\begin{enumerate}
\item If $\delta>0$, then the Poisson--Furstenberg boundary is isomorphic to $(\partial T ,\mathcal{B}_{\partial T},\nu_{\partial T})$ endowed with the boundary map $\bnd_{\partial T}:\Omega\to\partial T$.
\item If $\delta<0$, then it is isomorphic to $(\partial T\times\mathbb{R},\mathcal{B}_{\partial T\times\mathbb{R}},\nu_{\partial T\times\mathbb{R}})$ endowed with $\bnd_{\partial T\times\mathbb{R}}:\Omega\to\partial T\times\mathbb{R}$.
\item If $\delta=0$ and $\ln(A_{X_{1}})$ has finite second moment and there is an $\varepsilon>0$ such that $\ln(1+|B_{X_{1}}|)$ has finite $(2+\varepsilon)$-th moment, then it is isomorphic to~$(\partial T ,\mathcal{B}_{\partial T},\nu_{\partial T})$ endowed with $\bnd_{\partial T}:\Omega\to\partial T$.
\end{enumerate}  
\end{thm}


\subsection{Action by isometries on the Euclidean plane}

Let us now assume that $G=\BS(p,q)$ with $1<p=|q|$. Again, the definition of the tree $T$ and the level function $\lambda$ remain the same and Remark \ref{rem:neighbours} can be adapted. However, the situation differs fundamentally from the ones discussed so far because each brick, see \circled{1} and \circled{2} in Figure~\ref{fig:hyperbolic}, would now have equally many edges on its upper and lower level. Therefore, we shall use the Euclidean plane $\mathbb{R}^{2}$ instead of the hyperbolic plane~$\mathbb{H}$. In order to construct a projection $\pi_{\mathbb{R}^{2}}:G\to\mathbb{R}^{2}$, consider the map $\pi_{\isom(\mathbb{R}^{2})}:\sset{a,b}\to\isom(\mathbb{R}^{2})$ given by
\begin{equation*}
  \pi_{\isom(\mathbb{R}^{2})}(a):=\left\{\begin{array}{ll}((x,y)\mapsto(x,y+1))&\text{if}~q>0 \\[2pt] ((x,y)\mapsto(-x,y+1))&\text{if}~q<0 \end{array}\right.\quad\text{and}\quad\pi_{\isom(\mathbb{R}^{2})}(b):=((x,y)\mapsto(x+1,y))\,.
\end{equation*}
In both cases, $q>0$ and $q<0$, we may apply von Dyck's theorem to extend the map uniquely to a group homomorphism $\pi_{\isom(\mathbb{R}^{2})}:G\to\isom(\mathbb{R}^{2})$. Now, we define $\pi_{\mathbb{R}^{2}}:G\to\mathbb{R}^{2}$ by $\pi_{\mathbb{R}^{2}}(g):=\pi_{\isom(\mathbb{R}^{2})}(g)(0,0)$. Note that, instead of the discrete hyperbolic plane, we obtain a discrete Euclidean plane $\Gamma_{v}$.

We want to show that, as soon as the projections converge to a random element in $\partial T$, independently of the vertical drift, the Poisson--Furstenberg boundary is isomorphic to $(\partial T ,\mathcal{B}_{\partial T},\nu_{\partial T})$. In particular, we do not need to introduce any boundary to capture the behaviour of the projections $\pi_{\mathbb{R}^{2}}(Z_{n})$.

Even though the action of the group $G$ on the tree $T$ is not faithful anymore, the proof of Lemma~\ref{lem:anytree} still shows that the projections $\pi_{T}(Z_{n})$ converge a.\,s.~to a random end in $\partial T$. As in the first assertion of Lemma \ref{lem:minimal}, we can show that the $G$-action on $\partial T$ is topologically minimal, whence the hitting measure~$\nu_{\partial T}$ is non-atomic and has full support. This finally allows us to prove the following version of the identification theorem.

\begin{thm}[``identification theorem'' for $\mathbf{1<p=|q|}$]
Let $Z=(Z_{0},Z_{1},\dotsc)$ be a random walk on a non-amenable Baumslag--Solitar group $G=\BS(p,q)$ with $1<p=|q|$ and increments $X_{1},X_{2},\dotsc$ of finite first moment. Then, the Poisson--Furstenberg boundary is isomorphic to $(\partial T ,\mathcal{B}_{\partial T},\nu_{\partial T})$ endowed with the boundary map $\bnd_{\partial T}:\Omega\to\partial T$.
\end{thm}

\vspace{-10pt}\begin{proof}[Proof sketch.]
As in the proof of Theorem \ref{thm:identification}, we take the $\mu$-boundary $(\partial T ,\mathcal{B}_{\partial T},\nu_{\partial T})$ and the $\check{\mu}$-boundary $(\partial T ,\mathcal{B}_{\partial T},\hspace*{0.9pt}\check{\hspace*{-0.9pt}\nu}_{\partial T})$. Then, we define gauges
\begin{equation*}
  \mathcal{G}_{k}:=\left\{\,g\in G\,:\,\dtree(\pi_{T}(1),\pi_{T}(g))\leq k~\text{and}~\deucl(\pi_{\mathbb{R}^{2}}(1),\pi_{\mathbb{R}^{2}}(g))\leq k\right\}\,.
\end{equation*}
Again, $\hspace*{0.9pt}\check{\hspace*{-0.9pt}\nu}_{\partial T}\otimes\nu_{\partial T}$-almost every pair of points $(\xi_{-},\xi_{+})\in\partial T \times\partial T$ has distinct ends $\xi_{-},\xi_{+}\in\partial T$, which we may connect by a doubly infinite reduced path $v:\mathbb{Z}\to T$. Let $\mathcal{S}(\xi_{-},\xi_{+})$ be the full $\pi_{T}$-preimage of the path, i.\,e.~the set consisting of all $g\in G$ such that $\pi_{T}(g)$ is contained in $v(\mathbb{Z})$. To all remaining pairs we assign the whole of $G$ as a strip. This way, the map $\mathcal{S}$ becomes measurable, $G$-equivariant, and satisfies the inequality of Remark \ref{rem:removeg}. Now, pick an arbitrary pair $(\xi_{-},\xi_{+})\in\partial T \times\partial T$ with distinct ends $\xi_{-},\xi_{+}\in\partial T$. We claim that
\begin{equation*}
  \frac{1}{n}\cdot\ln\left(\card\left(\mathcal{S}(\xi_{-},\xi_{+})\cap\mathcal{G}_{|Z_{n}|}\right)\right)\leq\frac{\ln\big((2\cdot|Z_{n}|+1)\cdot(2\cdot|Z_{n}|+1)\big)}{n}\,.
\end{equation*}
Indeed, the strip $\mathcal{S}(\xi_{-},\xi_{+})$ intersects the gauge $\mathcal{G}_{|Z_{n}|}$ in at most $2\cdot|Z_{n}|+1$ many cosets from $G/B$, and each of them contains at most $2\cdot|Z_{n}|+1$ many elements of the gauge. Now, it suffices to consider the standard generating set $S:=\sset{a,b}\subseteq G$ and to observe that $|Z_{n}|\leq\dwordS(1,Z_{n})+1$. Then, as in the proof of Theorem~\ref{thm:identification}, we may use the fact that $\frac{1}{n}\cdot\dwordS(1,Z_{n})$ is a.\,s.~bounded and conclude that
\begin{equation*}
  \dotsc=\frac{\ln\big((2\cdot|Z_{n}|+1)^{2}\big)}{n}\leq\frac{\ln\big((2\cdot\dwordS(1,Z_{n})+3)^{2}\big)}{n}\,\xrightarrow[\text{\raisebox{.5ex}{a.\,s.}}]{~n\to\infty~}\,0\,,
\end{equation*}
which allows us to apply the strip criterion.
\end{proof}


\renewcommand{\leftmark}{REFERENCES}

\small

\bibliographystyle{amsalpha}

\providecommand{\bysame}{\leavevmode\hbox to3em{\hrulefill}\thinspace}

 \medskip

\textsc{Johannes Cuno, D\'epartement de math\'ematiques et applications, \'Ecole normale sup\'erieure, CNRS, PSL Research University, 45 rue d'Ulm, 75005 Paris, France} --- \texttt{johannes.cuno@ens.fr} \medskip \\
\textsc{Ecaterina Sava-Huss, Institute of Discrete Mathematics, Graz University of Technology, \\ Steyrergasse 30\,/\,III, 8010 Graz, Austria} --- \texttt{sava-huss@tugraz.at}

\end{document}